\documentclass[11pt]{article}%
\usepackage[latin9]{inputenc}
\usepackage{amsmath}
\usepackage{amssymb}
\usepackage{color}
\usepackage{amsfonts}
\usepackage{amsthm}
\usepackage{latexsym}
\usepackage{graphicx}
\usepackage{subfigure}
\usepackage{cite}%
\providecommand{\U}[1]{\protect\rule{.1in}{.1in}}

\makeatletter
\newtheorem{theorem}{Theorem}

\newtheorem{remark}[theorem]{Remark}
\newtheorem{corollary}[theorem]{Corollary}
\newtheorem{proposition}[theorem]{Proposition}

\begin{document}

\title{Theoretically exact photoacoustic reconstruction from spatially and temporally
reduced data}
\author{Ngoc Do
\and Leonid Kunyansky}
\maketitle

\begin{abstract}
We investigate the inverse source problem for the wave equation, arising in
photo- and thermoacoustic tomography. There exist quite a few theoretically
exact inversion formulas explicitly expressing solution of this problem in
terms of the measured data, under the assumption of constant and known speed
of sound. However, almost all of these formulas require data to be measured
either on an unbounded surface, or on a closed surface completely surrounding
the object. This is too restrictive for practical applications. The
alternative approach we present, under certain restriction on geometry, yields
theoretically exact reconstruction of the standard Radon projections of the
source from the data measured on a finite open surface. In addition, this
technique reduces the time interval where the data should be known. In
general, our method requires a pre-computation of densities of certain
single-layer potentials. However, in the case of a truncated circular or
spherical acquisition surface, these densities are easily obtained
analytically, which leads to fully explicit asymptotically fast algorithms. We
test these algorithms in a series of numerical simulations.

\end{abstract}


\section{Introduction}

We consider the inverse source problem arising in thermo- and photoacoustic
tomography (TAT/PAT)\cite{Kruger,Ora,Kruger1}. There exists a variety of
explicit inversion formulas that solve this problem under certain simplifying
conditions. As data, they use the time-dependent values of acoustic pressure
measured on a surface surrounding the object of interest. However, these
formulas require more data (both in time and in space) than is practical to
measure and is necessary from the theoretical point of view. The goal of this
paper is to develop explicit, theoretically exact formulas for solving the
inverse source problem of TAT/PAT, using sets of data reduced both spatially
and temporally.

TAT and PAT are novel coupled-physics modalities, designed to combine high
resolution of ultrasound techniques with high sensitivity of electromagnetic
waves to electrical and optical properties of biological tissues. In
PAT\cite{Kruger,Ora}, the region of interest (e.g., woman's breast in
mammography) is irradiated with a short laser pulse. In TAT\cite{Kruger1}, a
pulse of microwave radiation is used instead. In both cases, radiation is
partially absorbed by tissues. This raises the temperature of the medium and,
due to thermoelastic expansion, triggers an acoustic wave that is measured by
transducers on the boundary of the object. The propagation of the pressure
wave $p(t,x)$ can be modeled by the wave equation
\begin{equation}%
\begin{cases}
p_{tt}=c^{2}(x)\Delta p,\quad t\geq0,\quad x\in\mathbb{R}^{n}\\
p(0,x)=f(x),\quad p_{t}(0,x)=0,
\end{cases}
\label{E:wave-eq}%
\end{equation}
where $f(x)$ is the initial pressure in the tissues, $c(x)$ is the speed of
sound, and $n$ is the dimension of the space, $n=2,3$. This simplified model
assumes that the absorption of the radiation happens instantaneously, and that
the acoustic wave propagates in the open space, i.e., without reflecting from
transducers or other parts of the acquisition scheme. It also neglects
absorption and dispersion of acoustic waves in tissues.

Let us assume that the support $\Omega_{0}$ of $f(x)$ lies inside a larger,
open, bounded, simple-connected region $\Omega^{-}$ with a smooth boundary
$\Gamma.$ The pressure is measured by transducers placed along an
\textbf{observation surface} $S\subset\Gamma$. This yields data $g(t,y)$:
\begin{equation}
g(t,y)\equiv p(t,y),\quad(t,y)\in\mathbb{R}^{+}\times S. \label{E:wave-data}%
\end{equation}
The\textbf{ inverse source problem} of PAT/TAT consists of reconstructing
initial pressure $f(x)$ from the measurements $g(t,y)$, $(t,y)\in
\mathbb{R}^{+}\times S.$ \ This problem received significant attention from
mathematicians in the recent years. Important theoretical and numerical
results were obtained concerning uniqueness and stability of the
reconstruction and design of efficient computational algorithms, see reviews
\cite{AKK,KuKu,KuKuRev} and references therein. A version of this problem that
accounts for multiple reflection of acoustic waves from the transducer arrays
or walls of the tank has been also considered, and a number of results were
obtained (see \cite{Kun-Lin} for results and references). For completeness,
one should also mention the important Quantitative PAT problem (see, e.g.
\cite{CoxQPAT,RenBal}), which relies on the solution of the inverse source
problem as the first step. Moreover, an inverse source problem of exactly the
same form as considered here, also arises in other hybrid modalities. For
example, Magnetoacoustoelectric, Acoustoelectric, and Ultrasound modulated
optical tomography utilize, as the first step, the so-called synthetic
focusing \cite{KuKuAET,Kun-MAET,KIW-eng}. This procedure is mathematically
equivalent to the above-mentioned problem of PAT/TAT.

Here we concentrate on the practically important case of constant sound speed.
This approximation is considered acceptable for wave propagation in soft
tissues, as, for example, in the case of breast imaging. Such a simple model
is frequently used by practitioners, since the position-dependent parameters
describing attenuation, diffraction, and dispersion of acoustic waves in a
particular patient are usually not available. On the other hand, due to
relative simplicity of the constant sound speed case, a significant number of
explicit inversion formulas
\cite{Anders,FPR,Rubin,MXW2,Salman,Nguyen,HP1,HP2,Nat12,Haltm-ellipse,Fawcett,Finch07,Den,Pala,PalaBook,Kun-corn,Kun-cube,LK2007b,Palam-limit}
and explicit series solutions
\cite{LK2007a,Kun-cyl,Haltm-series,Norton1,Norton2,Amb1,Amb2} were obtained in
recent years.

In the case of constant sound speed, without loss of generality one may assume
$c(x)\equiv1.$ Then%
\begin{equation}
g(t,y)=\frac{\partial}{\partial t}G(t,y),\qquad G(t,y)\equiv\int%
\limits_{\Omega_{0}}f(x)\Phi_{n}(t,x-y)dx, \label{E:Gr-conv}%
\end{equation}
where $\Phi_{n}(t,x)$ is the fundamental solution \ (retarded Green's
function) of the free-space wave equation
\begin{equation}
u_{tt}(t,x)=\Delta u(t,x),\qquad t\in\mathbb{R},\qquad x\in\mathbb{R}^{n}.
\label{E:wave-constant}%
\end{equation}
It is well known that%
\begin{equation}
\Phi_{2}(t,x)=\frac{H(t-|x|)}{2\pi\sqrt{t^{2}-|x|^{2}}},\qquad\Phi
_{3}(t,x)=\frac{\delta(t-|x|)}{4\pi|x|}, \label{E:Gr-funs}%
\end{equation}
where $\delta(t)$ is the Dirac's delta distribution, and $H(t)$ is the
Heaviside function, equal to 1 when $t>0,$ and equal to 0 otherwise.
Importantly, $\Phi_{n}(t,x)$ reflects the finite speed of propagation of the
sound waves:%
\begin{equation}
\Phi_{n}(t,x)=0\text{ if }t<|x|. \label{E:causal}%
\end{equation}
\qquad Using equations (\ref{E:Gr-conv}) and (\ref{E:Gr-funs}) one can relate
the data $g(t,y)$ to the integrals $I(r,y)$ of $f$ over spheres (circles) with
the centers lying on $S:$%
\begin{align}
g(t,y)  &  =\frac{\partial}{\partial t}\frac{I(t,y)}{4\pi t},\qquad n=3,\qquad
y\in S,\label{E:wave-to-int-3D}\\
g(t,y)  &  =\frac{\partial}{\partial t}\int\limits_{0}^{t}\frac{I(r,y)}%
{2\pi\sqrt{t^{2}-r^{2}}}dr,\qquad n=2,\qquad y\in S,\label{E:wave-to-int-2D}\\
I(r,y)  &  \equiv r^{n-1}\int\limits_{\mathbb{S}^{n-1}}f(y+r\hat{l})d\hat{l}.
\label{E:def-int}%
\end{align}
If $\Omega^{-}$ is bounded, the support of $r \mapsto I(r,y)$ ($y \in S$ is
fixed) is contained in the interval $M\equiv\lbrack0,\mathrm{diam}(\Omega
^{-})],$ $n=2,3.$ In $\mathbb{R}^{3}$ (3D), due to (\ref{E:wave-to-int-3D})
wave data $g(t,y)$ are also supported in $t$ within $M$ -- this is
manifestation of the Huygens' principle. Finding $I(r,y)$ from $g(t,y)$ is
trivial in this case. In $\mathbb{R}^{2}$ (2D), $t \mapsto g(t,y)$ is, in
general, not compactly supported in $\mathbb{R}$. Integrals $I(r,y)$ can be
recovered from $g(t,y)$ in 2D by inverting the Abel transform
(\ref{E:wave-to-int-2D}):%
\begin{equation}
I(r,y)=4\int\limits_{0}^{r}\frac{g(t,y)}{\sqrt{r^{2}-t^{2}}}dt,\qquad
n=2,\qquad y\in S,\qquad r\in M. \label{E:inverse-Abel}%
\end{equation}
We see that, for a fixed $y\in S,$ explicit recovery of all values of $I(r,y)$
in $r$ requires only the knowledge of $g(t,y)$ for all $t\in M.$ Moreover, by
combining (\ref{E:inverse-Abel}) with (\ref{E:wave-to-int-2D}) one can
explicitly recover values of $g(t,y)$ for $t\in(\mathrm{diam}(\Omega
^{-}),\infty)$ from values of $g(t,y)$ for $t\in M.$

If $\Omega^{-}$ is unbounded, the supports of $r \mapsto I(r,y)$ and $t
\mapsto g(t,y)$ are unbounded in both 2D and 3D.

All of the above mentioned explicit inversion formulas for bounded domains
(with the exception of \cite{Amb1,Amb2}) require knowledge of $I(r,y)$ or
$g(t,y)$ on $M\times\Gamma$ in 3D. \ In 2D, the data are either $I(r,y)$ on
$M\times\Gamma$, or $g(t,y)$ on $\mathbb{R}^{+}\times\Gamma$; however, as
explained above, values of $g(t,y)$ on $(\mathrm{diam}(\Omega^{-}%
),\infty)\times\Gamma$ can be explicitly reconstructed from those known on
$M\times\Gamma.$ Explicit formulas for unbounded domains in 2D or 3D require
data ($I$ or $g)$ to be measured on infinite intervals in time or in radius.

\section{Formulation of the problem}

In practical applications, measurements cannot be performed over unbounded
surfaces. Therefore, inversion formulas that assume such data, have to be
applied to sets that are\emph{ truncated in space}. All known explicit
inversion formulas designed for bounded measurement surfaces (except
\cite{Palam-limit}) use data given on the full boundary $\Gamma.$
Unfortunately, in applications such data also have to be truncated. Indeed, in
most medical applications of PAT/TAT (e.g., in breast imaging) the region of
interest (ROI) is not the whole human body but rather a part of it. Then, one
cannot surround the ROI by the detectors from all sides. Spatial truncation of
data makes existing inversion formulas inexact and leads to significant
artifacts in the reconstructed images. We, thus, are looking for
reconstruction techniques that use data reduced spatially, i.e. measured on
proper subsets $S$ of $\Gamma.$

Perhaps the most important reason to use the data\emph{ truncated temporally}
(i.e., collected over a proper subset of $M$ in time or radius), is the
deterioration of acoustic waves as they propagate through the tissues. This
happens due to absorption, diffraction, and dispersion \cite{Anastas} that are
not reflected by the wave equation (\ref{E:wave-constant}). In order to reduce
the effect of such deterioration, practitioners (see e.g., \cite{Anastas}) use
the so-called "half-time reconstruction" that consists of simply truncating
the data in time and applying one of the known inversion formulas (e.g., the
\textquotedblleft universal" backprojection formula \cite{MXW2}). Such
truncation reduces some artifacts, but introduces new errors, since the
formula applied to the partial data is no longer exact.

From the theoretical standpoint, the inverse source problem of TAT/PAT can be
solved with significantly less data. Indeed, it is known that this problem is
stably solvable if the acquisition surface $S$ and the support $\Omega_{0}$ of
$f$ satisfy the \textbf{visibility condition}
\cite{XWAK,XWAK2,KuKu,KuKuRev,FPR,StUhl}. In the case of the constant speed of
sound this condition requires that for each point $x\in\Omega_{0}$ and for
each direction $l,$ a straight line passing through $x$ parallel to $l$ would
intersect $S$ at least once. Such a condition can be satisfied by a
significantly reduced set of measurements. For example, if $\Omega_{0}%
=\Omega^{-}$ is a ball of radius $R$ and $S$ is its boundary, the set of
spheres with centers on $S$ and radii in the interval $[0,R]$ satisfy the
visibility condition with respect to $\Omega_{0}$, and thus the inverse
problem can be solved with data $I(r,y)$ or $g(t,y)$ given on a temporally
reduced set $[0,R]\times S$. Another example is the problem with $\Omega
_{0}=\Omega^{-}$ being a lower half of a ball of radius $R,$ and $S$ a
concentric lower half-sphere of the same radius. Then the visibility condition
is satisfied for radii covering the range $[0,2R]$ and centers restricted to
the open surface $S.$ The known exact inversion formulas, however, can only
work with data measured either on a whole sphere, or on a full boundary of a half-sphere.

The only existing exact solution of the inverse source problem of PAT/TAT with
temporally reduced data was obtained in \cite{Amb1,Amb2,Amb-Roy}. There, the
problem was solved for circular, spherical, and elliptical domains by
expanding the function and the data in circular/spherical harmonics, and
deriving Cormack-like formulas. However, the resulting formulas contain
convolutions with high-order Chebyshev polynomials. Such formulas can be
difficult to accurately implement numerically; moreover, similarly to the
original Cormack's formulas (see e.g. \cite{Natt4}) the inversion is
ill-posed. Numerical implementation based on numerical inversion of the
arising Volterra integral equation was reported in \cite{Amb-Roy} for the case
of spherical acquisition surface in 3D.

An example of a theoretically exact reconstruction technique with a spatially
reduced set is given in \cite{Palam-limit}, where a 2D problem is solved with
$\Omega_{0}$ being a half disk and $S$ its diameter. This combination of
acquisition surface and support of the source does not satisfy the visibility
condition, so the reconstruction procedure is exponentially unstable.

Other techniques, proposed for the problem with a reduced set of data in the
past are either approximate \cite{Kun-open,PS1,PS2}, or iterative, such as,
e.g. \cite{Stef-Neumann}. The latter method consists of treating the distorted
image as a first approximation and then refining it iteratively. Such an
algorithm yields good results in 2D simulations~\cite{Stef-Neumann} based on
the finite-difference time reversal. However, convergence of the corresponding
Neumann series has not been proven theoretically. Moreover, on large
computational grids in 3D such an algorithm would require many hours of
computation, having the complexity of each iteration $\mathcal{O}(m^{4})$
flops for a grid of size $m\times m\times m.$

In this paper, instead of reconstructing $f(x)$ directly, we solve a problem
of reconstructing the standard Radon projections of $f(x)$ from data reduced
spatially and/or temporally. Reconstructing function $f(x)$ thus can be
completed by inverting the standard Radon transform numerically. However, the
problem of recovering a function from the full set of its Radon projections is
very well understood by now, and a variety of accurate and efficient numerical
algorithms are well known\cite{Natt4}, so we will simply omit the discussion
of it. Our technique is based on representation of plane waves within a region
by a single layer potential supported on the region's boundary$.$ Such
representations are discussed in Section~\ref{S:repres}. The rest of the
papers is organized as follows. Reconstruction of the Radon projections of a
function from spatially and/or temporally reduced data for general acquisition
surfaces is presented in Section \ref{S:Recovery}. The special cases of
truncated circular (in 2D) and spherical (in 3D) acquisition surfaces are
considered in Section \ref{S:Spheres}. The latter section also contains the
description of efficient numerical algorithms and results of numerical
simulations demonstrating the work of these techniques.

\section{Representing plane waves by single layer potentials}

\label{S:repres}

The goal of this section is to represent a propagating delta wave
$\delta(-\omega\cdot x+t)$ (where $\omega$ is a unit vector) by a single layer
potential. Without loss of generality, let us assume that the smallest ball
containing $\Omega^{-}$ is found, and that $\Omega^{-}$ is translated so that
this ball of radius $R\ $is centered at the origin; we will denote it by
$B(0,R)$ (see Figure~\ref{F:support}(a)). \begin{figure}[h]
\begin{center}%
\begin{tabular}
[c]{cc}%
\includegraphics[scale = 0.9]{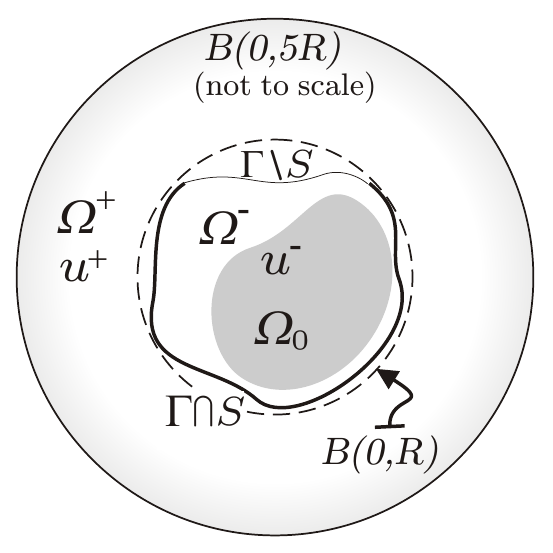} &
\includegraphics[scale = 0.9]{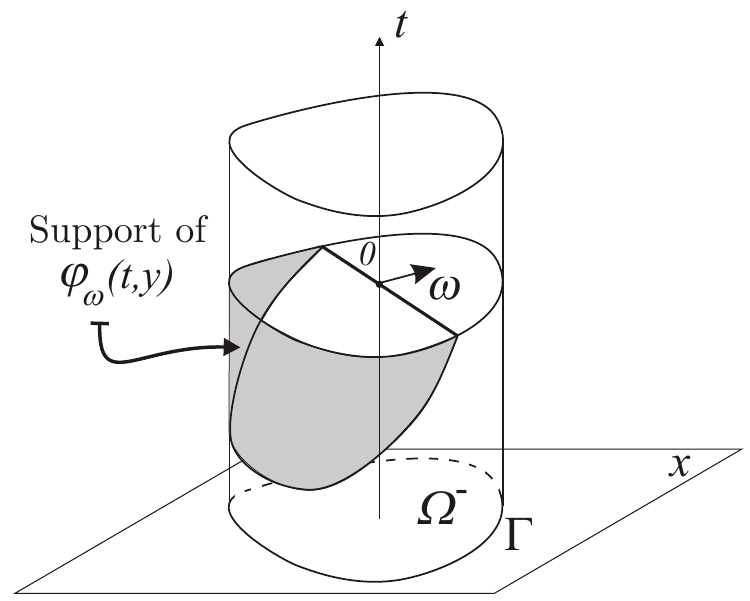}\\
(a) & (b)\\
&
\end{tabular}
\end{center}
\par
\vspace{-8mm}\caption{(a) geometry of the problem (b) support of density
$\varphi_{\omega}(t,y)$ }%
\label{F:support}%
\end{figure}The support of the wave $\delta(-\omega\cdot x+t)$ is a hyperplane
$L(t,\omega)$ given by equation $t=\omega\cdot x.$ Then, $L(t,\omega)$
intersects $\Omega^{-}$ for all values of $t$ lying in the time interval
$\mathcal{T}(\omega)\equiv$ $(T_{0}(\omega),T_{1}(\omega))$ with $-R\leq
T_{0}(\omega)<T_{1}(\omega)\leq R.$ We note for future use the equalities
\begin{equation}
T_{1}(-\omega)=-T_{0}(\omega),\qquad T_{0}(-\omega)=-T_{1}(\omega).
\label{E:endpoints}%
\end{equation}
We would like to represent wave $\delta(-\omega\cdot x+t)$ by the following
retarded potential%
\[
\delta(-\omega\cdot x+\tau)=\int\limits_{T_{0}(\omega)}^{\tau}\int%
\limits_{\Gamma}\Phi_{n}(\tau-t,x-y)\varphi_{\omega}(t,y)dydt,\qquad
x\in\Omega^{-},\qquad\tau\in(T_{0}(\omega),0],
\]
where $\varphi_{\omega}(t,y)$ is a certain distribution supported on
$\mathcal{T}(\omega)\times\Gamma.$ One way of doing this is through a solution
of a surface scattering problem for the wave equation as described below. We
will investigate the latter problem following \cite{Sayas}; we refer to this
book for precise definitions of spaces and distributions we use below.

From now on we will assume that the boundary $\Gamma$ of $\Omega^{-}$ is
Lipschitz (see \cite{Sayas} for the precise definition). Let us denote by
$\Omega^{+}$ the complement to the closure of~$\Omega^{-}$, i.e. $\Omega
^{+}\equiv\mathbb{R}^{n}\backslash\overline{\Omega^{-}}.$ Consider a plane
wave $\mu_{\omega}(t,x)$ of the form
\begin{equation}
\mu_{\omega}(t,x)\equiv s(\omega\cdot x-t), \label{E:planewave}%
\end{equation}
where $s:\mathbb{R}\rightarrow\mathbb{R}$ is a continuously differentiable and
piece-wise twice-differentiable function on $\mathbb{R}$, vanishing on the
interval $(0,+\infty)$%
\begin{equation}
s(t)=0,\qquad t\in(0,+\infty),\qquad s\in C^{1}(\mathbb{R})\cap C^{2}%
(\text{a.e. on }\mathbb{R}). \label{E:finitesp}%
\end{equation}
Let us restrict our attention to a large bounded domain containing all our
objects; as such, we will use the ball $B(0,5R)$ of radius $5R$ centered at
the origin (see figure~\ref{F:support}(a)). For each fixed $t,$ function
$\mu_{\omega}(t,x)$ belongs to the Sobolev class $H^{2}(B(0,5R))$ with
$\Delta\mu_{\omega}\in L^{2}(B(0,5R))$. Then, there exists (see \cite{Sayas},
section 3.3)\ a unique tempered distribution $\xi_{\omega}(t,\cdot)$ with
values in $H^{-1/2}(\Gamma)$, with $t\mapsto\xi_{\omega}(t,\cdot)$ supported
on $\mathcal{T}(\omega),$ and such that the single layer potential
\[
\left(  S\ast\xi_{\omega}\right)  (\tau,x)\equiv\int\limits_{T_{0}(\omega
)}^{\tau}\int\limits_{\Gamma}\Phi_{n}(\tau-t,x-y)\xi_{\omega}(t,y)dydt,\qquad
x\in\mathbb{R}^{n}\backslash\Gamma,\qquad\tau\in\mathcal{T}(\omega),
\]
defines distributional solutions $u^{+}$ and $u^{-}$ of the wave equation
$\frac{\partial^{2}}{\partial t^{2}}u^{\pm}(t,x)=\Delta u^{\pm}(t,x)$ on
$\Omega^{+}$ and on $\Omega^{-}$ as%
\[
u^{\pm}(t,x)=\left(  S\ast\xi_{\omega}\right)  (t,x),\qquad x\in\Omega^{\pm}.
\]
Both $u^{+}$ and $u^{-}$ are tempered distributions in $t$ with values in
$H_{\Delta}^{1}(\Omega^{\pm}),$ where $H_{\Delta}^{1}(\Omega^{\pm})$ is the
space of functions from $H^{1}(\Omega^{\pm})$ whose Laplacian is in
$L^{2}(\Omega^{\pm}).$ Moreover, $\xi_{\omega}$ is such that
\begin{equation}
\mu_{\omega}(t,x)=u^{-}(t,x)\equiv\int\limits_{T_{0}(\omega)}^{\tau}%
\int\limits_{\Gamma}\Phi_{n}(\tau-t,x-y)\xi_{\omega}(t,y)dydt,\qquad
x\in\Omega^{-},\qquad\tau\in\mathcal{T}(\omega), \label{E:single-layer}%
\end{equation}
where the equality is not point-wise, but is understood in the sense of
tempered distributions in $t$ with values in $H_{\Delta}^{1}(\Omega^{-}).$

On the other hand, distribution \ $u^{+}(t,x)$ solves the problem of soft
scattering of the incoming wave $\mu_{\omega}(t,x)$ by the surface $\Gamma.$
Indeed, following \cite{Sayas}, let us define the exterior and interior trace
operators $\gamma^{\pm}:H_{\Delta}^{1}(\mathbb{R}^{n}\backslash\Gamma
)\rightarrow H^{1/2}(\Gamma)$ and exterior and interior normal derivatives
$\partial_{n}^{\pm}:H_{\Delta}^{1}(\mathbb{R}^{n}\backslash\Gamma)\rightarrow
H^{-1/2}(\Gamma).$ Then
\[
\gamma^{+}u^{+}(t,\cdot)=-\gamma^{+}\mu_{\omega}(t,\cdot),\qquad
t\in\mathcal{T}(\omega),
\]
so that the total field $u^{tot}(t,x)=\mu_{\omega}(t,x)+u^{+}(t,x)$ satisfies
the zero Dirichlet condition on $\Gamma.$ Importantly, the jump of the normal
derivatives across $\Gamma$ satisfies the following jump relation (in the
sense of $H^{-1/2}(\Gamma)$ distributions):%
\begin{equation}
\partial_{n}^{-}\mu_{\omega}(t,\cdot)-\partial_{n}^{+}u^{+}(t,\cdot)=\mathrm{
}\xi_{\omega}(t,\cdot),\qquad t\in\mathcal{T}(\omega). \label{E:jump}%
\end{equation}

The relation between the single layer representation of the wave $\mu_{\omega
}(t,x)$ in $\Omega^{-},$ and the solution of the scattering problem
$u^{+}(t,x)$ leads to the following important observation on support of the
density $\mathrm{ }\xi_{\omega}(t,\cdot).$ Due to (\ref{E:finitesp}),\ the
incoming wave $\mu_{\omega}(t,x)$ vanishes in the region $W$ consisting of
points with $\omega\cdot x>t$:%
\begin{equation}
\mu_{\omega}(t,x)=0,\quad\forall(t,x)\in W,\quad W\equiv\{(t,x)|\ t\in
\mathcal{T}(\omega),\ x\in\mathbb{R}^{n},\ \omega\cdot x>t\}. \label{E:def-W}%
\end{equation}
Due to the finite speed (equal to 1) of propagation of waves governed by the
wave equation, total field $u^{tot}$ (and, hence, scattered wave $u^{+}$)
vanishes\footnote{Vanishing of $u^{tot}(t,x)$ in $W^{+}$ can be proven
formally by considering the energy $\mathcal{E}(t)$ within the region
$W_{B}^{+}\equiv W^{+}\cap(\mathbb{R}\times B(0,5R)),$ defined as
$\mathcal{E}(t)\equiv\int_{W_{B}^{+}}\left[  (\nabla u^{tot})^{2}+\left(
\frac{\partial}{\partial t}u^{tot}\right)  ^{2}\right]  dxdt.$ Note that
$\mathcal{E}(T_{0}(\omega))=0$. The only possible source of energy in
$W_{B}^{+}$ is the part of the boundary $\partial_{0}W\equiv W\cap
(\mathbb{R\times}\Gamma).$ However, since $u^{tot}(t,x)$ vanishes on
$\partial_{0}W$, the flux of energy through $\partial_{0}W$ is zero, and
$\mathcal{E}(t)=0$ for all $t\in\mathcal{T}(\omega)$.} in $W^{+}\equiv
W\cap(\mathcal{T}(\omega)\times\Omega^{+})$. Therefore, jump condition
(\ref{E:jump}) implies that
\begin{equation}
\mathrm{ }\xi_{\omega}(t,x)=0,\qquad\forall(t,x)\in W\cap(\mathbb{R\times
}\Gamma). \label{E:support-psi}%
\end{equation}
Moreover, due to uniqueness of $\mathrm{ }\xi_{\omega}(t,x)$ for a given
$\mu_{\omega}(t,x),$ support condition (\ref{E:support-psi}) holds
independently of the way density $\mathrm{ }\xi_{\omega}(t,x)$ was found. We
summarize these results in

\begin{proposition}
\label{T:exist}Given a Lipschitz domain $\Omega^{-}$ and plane wave
$\mu_{\omega}(t,x)$ satisfying (\ref{E:planewave}), (\ref{E:finitesp}), there
exists a unique tempered distribution $\mathrm{ }\xi_{\omega}(t,\cdot)$
supported on $\mathcal{T}(\omega)$ with values in $H^{-1/2}(\Gamma),$ such
that $\mu_{\omega}(t,x)$ is represented in $\mathcal{T}(\omega)\times
\Omega^{-}$ by the single layer potential in the form (\ref{E:single-layer}).
Moreover, density $\mathrm{ }\xi_{\omega}$ vanishes on $W\cap(\mathbb{R\times
}\Gamma),$ where $W$ is defined by (\ref{E:def-W}).
\end{proposition}

Let us now consider a particular choice of wave $\mu_{\omega}(t,x).$ Let us
define function $\tau_{-}^{2}$ as follows
\begin{equation}
\tau_{-}^{2}=\left\{
\begin{array}
[c]{cc}%
0, & \tau\geq0,\\
\tau^{2}, & \tau<0,
\end{array}
\right.  \label{E:special}%
\end{equation}
and choose $\mu_{\omega}(t,x)\equiv\frac{1}{2}(\omega\cdot x-t)_{-}^{2}$.
According to Proposition \ref{T:exist}, there is density $\xi_{\omega}$ \ such
that
\begin{equation}
\frac{1}{2}(\omega\cdot x-\tau)_{-}^{2}=\int\limits_{T_{0}(\omega)}^{\tau}%
\int\limits_{\Gamma}\Phi_{n}(\tau-t,x-y)\xi_{\omega}(t,y)dydt,\qquad
x\in\Omega^{-},\qquad\tau\in\mathcal{T}(\omega), \label{E:delta-layer}%
\end{equation}
with equality understood in the sense of distributions. We will extend
$\xi_{\omega}$ by 0 to -$\infty$ and use (\ref{E:causal}) to extend the
integration interval to $(-\infty,T_{1}(\omega))$:%
\begin{equation}
\frac{(\omega\cdot x-\tau)_{-}^{2}}{2}=\int\limits_{-\infty}^{T_{1}(\omega
)}\int\limits_{\Gamma}\Phi_{n}(\tau-t,x-y)\xi_{\omega}(t,y)dydt=\int%
\limits_{\tau-T_{1}(\omega)}^{\infty}\int\limits_{\Gamma}\Phi_{n}%
(s,x-y)\xi_{\omega}(\tau-s,y)dyds, \label{E:work}%
\end{equation}
\bigskip where $x\in\Omega^{-}$ and $\tau\in\mathcal{T}(\omega).$
Differentiating (\ref{E:work})\ thrice in $\tau$ and taking into account
(\ref{E:causal}) yields:%
\begin{equation}
\delta(\tau-\omega\cdot x)=\int\limits_{\tau-T_{1}(\omega)}^{\infty}%
\int\limits_{\Gamma}\Phi_{n}(s,x-y)\frac{\partial^{3}}{\partial\tau^{3}}%
\xi_{\omega}(\tau-s,y)dyds=\int\limits_{T_{0}(\omega)}^{\tau}\int%
\limits_{\Gamma}\Phi_{n}(\tau-t,x-y)\varphi_{\omega}(t,y)dydt,
\label{E:single-delta}%
\end{equation}
with $x\in\Omega^{-},$ $\tau\in\mathcal{T}(\omega),$ where $\varphi_{\omega
}(t,y)$ is defined as the following distributional derivative:
\begin{equation}
\varphi_{\omega}(t,\cdot)\equiv\frac{\partial^{3}}{\partial t^{3}}\xi_{\omega
}(t,\cdot),\qquad t\in\mathcal{T}(\omega). \label{E:third-der}%
\end{equation}
Importantly, $\varphi_{\omega}(t,y)$ has the same support (in the sense of
distributions) as $\xi_{\omega}(t,y),$ i.e.
\begin{equation}
\varphi_{\omega}(t,y)=0\text{ on }W\cap(\mathbb{R\times}\Gamma).
\label{E:small_support}%
\end{equation}
We thus have proven

\begin{proposition}
\label{T:exist-Heavi}Given a Lipschitz domain $\Omega^{-},$ there exists a
unique distribution $\varphi_{\omega}(t,y)$ defined by equations
(\ref{E:special}), (\ref{E:delta-layer}), (\ref{E:third-der}) such that delta
wave $\delta(-\omega\cdot x+\tau)$ is represented in $\mathcal{T}%
(\omega)\times\Omega^{-}$ by the single layer potential in the form
(\ref{E:single-delta}). Moreover, density $\varphi_{\omega}(t,y)$ vanishes on
$W\cap(\mathbb{R\times}\Gamma),$ where $W$ is defined by (\ref{E:def-W}).
\end{proposition}

The support of $\varphi_{\omega}(t,y)$ is shown in Figure \ref{F:support}(b).

\section{Recovering Radon projections from thermoacoustic
data\label{S:Recovery}}

Our approach to the inverse source problem of TAT/PAT consists of finding the
standard Radon projections $\mathcal{R}f(\omega,\tau)$ of the source $f(x)$
through the use of single layer potential(s) for the wave equation. The latter
projections are defined as follows%
\[
\mathcal{R}f(\tau,\omega)\equiv\int\limits_{\Omega_{0}}f(x)\delta(\tau
-\omega\cdot x)dx,
\]
where $\omega\in\mathbb{S}^{n-1}$ is a unit vector. Since explicit formulas
and efficient algorithms for inverting the Radon transform are well known
\cite{Natt4}, recovering $\mathcal{R}f(\omega,\tau)$ is equivalent to finding
$f(x).$

For convenience, let us extend $G(t,y)$ to $\mathbb{R}\times\Gamma$ by zero%
\[
G(t,y)\equiv0,\quad t\in(-\infty,0),\quad y\in\Gamma.
\]
We assume, for simplicity, that $f(x)\in C_{0}^{\infty}(\Omega_{0}),$ making
$G(t,y)$ an infinitely differentiable function of $t.$ Let us multiply
$G(t,y)$ by $\varphi_{\omega}(\tau-t,y)$ (as given by Proposition
\ref{T:exist-Heavi}), and integrate over $(0,\tau-T_{0}(\omega)]\times\Gamma$:%
\begin{align}
\int\limits_{0}^{\tau-T_{0}(\omega)}\int\limits_{\Gamma}G(t,y)\varphi_{\omega
}(\tau-t,y)dydt  &  =\int\limits_{0}^{\tau-T_{0}(\omega)}\int\limits_{\Gamma
}\left[  \int\limits_{\Omega_{0}}f(x)\Phi_{n}(t,x-y)dx\right]  \varphi
_{\omega}(\tau-t,y)dydt\label{E:findingJ1}\\
&  =\int\limits_{\Omega_{0}}f(x)\left[  \int\limits_{T_{0}(\omega)}^{\tau}%
\int\limits_{\Gamma}\Phi_{n}(\tau-s,x-y)\varphi_{\omega}(s,y)dyds\right]
dx=\nonumber
\end{align}%
\begin{equation}
=\int\limits_{\Omega_{0}}f(x)\delta(-\omega\cdot x+\tau)dx=\mathcal{R}%
f(\tau,\omega),\quad\omega\in\mathbb{S}^{n-1},\quad\tau\in\mathcal{T}(\omega).
\label{E:findingJ}%
\end{equation}
Thus, we have proven

\begin{theorem}
\label{T:rec-form}Let $f\in C_{0}^{\infty}(\Omega_{0}),$ data $G(t,y)$ be
defined by (\ref{E:Gr-conv}), and distribution $\varphi_{\omega}(t,y)$ be
given by Proposition \ref{T:exist-Heavi}. Then, for any $\omega\in
\mathbb{S}^{n-1}$ and $\tau\in\mathcal{T}(\omega),$ Radon projections
$\mathcal{R}f(\tau,\omega)$ can be reconstructed from $G(t,y)$ by the formula%
\begin{equation}
\mathcal{R}f(\tau,\omega)=\int\limits_{0}^{\tau-T_{0}(\omega)}\int%
\limits_{\Gamma}G(t,y)\varphi_{\omega}(\tau-t,y)dydt. \label{E:good-rec}%
\end{equation}

\end{theorem}

We note that for $\tau\notin\mathcal{T}(\omega)$, projections $\mathcal{R}%
f(\tau,\omega)$ vanish. Therefore, formula (\ref{E:good-rec}) recovers all
Radon projections $\mathcal{R}f(\tau,\omega)$ for $\omega\in\mathbb{S}^{n-1}$
and $\tau\in\mathbb{R}$ from the full set of data $G(t,y)$ defined on
$[0,\mathrm{diam}(\Omega^{-})]\times\Gamma.$

Let us extend $\varphi_{\omega}(t,y)$ by zero outside of the interval
$\mathcal{T}(\omega)$ in $t.$ Then, since $G(t,y)$ vanishes for $t<0,$ the
integral in $t$ in the left hand side of (\ref{E:findingJ1}) can be extended
to all of $\mathbb{R}$, making it a standard convolution in time:%
\begin{equation}
\mathcal{R}f(\tau,\omega)=\int\limits_{\mathbb{R}}\int\limits_{\Gamma
}G(t,y)\varphi_{\omega}(\tau-t,y)dydt,\quad\omega\in\mathbb{S}^{n-1},\quad
\tau\in\mathcal{T}(\omega). \label{E:convolution}%
\end{equation}
By differentiating the above equation in $\tau$ and integrating by parts one
obtains an expression for the derivative of the Radon projections in $\tau$ in
terms of $g(t,y).$

\begin{corollary}
\label{T:corr-g}Under the conditions of theorem \ref{T:rec-form} the following
formula holds for any $\omega\in\mathbb{S}^{n-1}$ and $\tau\in\mathcal{T}%
(\omega)$:%
\begin{equation}
\frac{\partial}{\partial\tau}\mathcal{R}f(\tau,\omega)=\int\limits_{0}%
^{\tau-T_{0}(\omega)}\int\limits_{\Gamma}g(t,y)\varphi_{\omega}(\tau
-t,y)dydt,\quad\omega\in\mathbb{S}^{n-1},\quad\tau\in\mathcal{T}(\omega).
\label{E:good-rec-g}%
\end{equation}

\end{corollary}

Due to the bounded support of $t \mapsto\varphi_{\omega}(t,y)$ (see
Proposition \ref{T:exist}) and the variable upper limit in the outer integrals
in (\ref{E:good-rec}) and (\ref{E:good-rec-g}) we obtain the following obvious
but important

\begin{remark}
\label{T:corr}Let $f\in C_{0}^{\infty}(\Omega_{0}),$ data $G(t,y)$ (or
$g(t,y)$) be defined by (\ref{E:Gr-conv}), and distribution $\varphi_{\omega
}(t,y)$ be given by Proposition \ref{T:exist-Heavi}. Suppose that instead of
data $G(t,y)$ we are given corrupted data $\widetilde{G}(t,y)$, such that for
some $T_{\mathrm{good}}<\mathrm{diam}(\Omega^{-})$, $\widetilde{G}(t,y)$
\ coincides with $G(t,y)$ on $(0,T_{\mathrm{good}}]\times\Gamma$ and differs
for larger values of $t.$ Then, for any $\omega\in\mathbb{S}^{n-1},$ Radon
projections $\mathcal{R}f(\tau,\omega)$ for values of $\tau$ in the interval
$(T_{0}(\omega),T_{0}(\omega)+T_{\mathrm{good}}],$ are reconstructed from
$\widetilde{G}(t,y)$ by the formula (\ref{E:good-rec}) with $G$ replaced
by~$\tilde{G}$. Similarly, the derivatives $\frac{\partial}{\partial\tau
}\mathcal{R}f(\tau,\omega)$ are reconstructed from $\widetilde{g}%
(t,y)\equiv\frac{\partial}{\partial\tau}\widetilde{G}(t,y)$ using equation
(\ref{E:good-rec-g}) with $g$ replaced by $\tilde{g}.$
\end{remark}

Formula (\ref{E:good-rec-g}) is more practical since usually $g(t,y)$ (and not
$G(t,y))$ is measured.

\subsection{Reconstruction from temporally reduced data}

Proposition \ref{T:rec-form} and Remark \ref{T:corr} also permit us to recover
the full set of projections $\mathcal{R}f(\tau,\omega)$ from temporally
reduced data. Indeed,\ due to the well known symmetry of the Radon
projections
\begin{equation}
\mathcal{R}f(\tau,\omega)=\mathcal{R}f(-\tau,-\omega), \label{E:Radon-symm}%
\end{equation}
equation (\ref{E:good-rec}) is over-determined. This allows for a temporal
reduction of data. Define $T_{\mathrm{med}}(\omega)$ as the middle of the
interval $\mathcal{T}(\omega)$ (i.e. $T_{\mathrm{med}}(\omega)\equiv
(T_{0}(\omega)+T_{1}(\omega))/2$) and notice that
\begin{equation}
T_{\mathrm{med}}(\omega)-T_{0}(\omega)\leq\frac{1}{2}\mathrm{diam}(\Omega
^{-}), \qquad T_{1}(\omega)-T_{\mathrm{med}}(\omega)\leq\frac{1}%
{2}\mathrm{diam}(\Omega^{-}). \label{E:silly-ineq}%
\end{equation}

\begin{theorem}
\label{T:reduced}\textbf{Reconstruction from temporally reduced data.} Let
$f\in C_{0}^{\infty}(\Omega_{0}),$ data $G(t,y)$ be defined by
(\ref{E:Gr-conv}), and distribution $\varphi_{\omega}(t,y)$ be given by
Proposition \ref{T:exist-Heavi}. Then, Radon projections $\mathcal{R}%
f(\tau,\omega)$ can be reconstructed from the data $G(t,y)$ known on $\left[
0,\frac{1}{2}\mathrm{diam}(\Omega^{-})\right]  \times\Gamma$ by the formulas%
\begin{align}
\mathcal{R}f(\tau,\omega)  &  =\int\limits_{0}^{\tau-T_{0}(\omega)}%
\int\limits_{\Gamma}G(t,y)\varphi_{\omega}(\tau-t,y)dydt,\qquad\tau\in
(T_{0}(\omega),T_{\mathrm{med}}(\omega)],\qquad\omega\in\mathbb{S}%
^{n-1},\label{E:thm1a}\\
\mathcal{R}f(\tau,\omega)  &  =\mathcal{R}f(-\tau,-\omega),\qquad\tau
\in(T_{\mathrm{med}}(\omega),T_{1}(\omega)),\qquad\omega\in\mathbb{S}^{n-1}.
\label{E:thm1b}%
\end{align}

\end{theorem}

\begin{proof}
Due to (\ref{E:silly-ineq}), for values of $\tau$ in the interval
$(T_{0}(\omega),T_{\mathrm{med}}(\omega)]$ the upper integration limit in $t$
in (\ref{E:thm1a}) never exceeds $\frac{1}{2}\mathrm{diam}(\Omega^{-})$.
Moreover, suppose $\tau\in(T_{\mathrm{med}}(\omega),T_{1}(\omega))$ as
required by formula~(\ref{E:thm1b}). Then, due to (\ref{E:endpoints}), the
value of $-\tau$ lies in the interval $(-T_{1}(\omega),-T_{\mathrm{med}%
}(\omega))=(T_{0}(-\omega),T_{\mathrm{med}}(-\omega)),$ and $\mathcal{R}%
f(-\omega,-\tau)$ can be computed using formula (\ref{E:thm1a}) with required
values of $G(t,y)$ confined to $t\in\left[  0,\frac{1}{2}\mathrm{diam}%
(\Omega^{-})\right]  $.
\end{proof}

Alternatively, the derivatives $\frac{\partial}{\partial\tau}\mathcal{R}%
f(\tau,\omega)$ can be reconstructed from $g(t,y)$ using the formula
\[
\frac{\partial}{\partial\tau}\mathcal{R}f(\tau,\omega)=\int\limits_{0}%
^{\tau-T_{0}(\omega)}\int\limits_{\Gamma}g(t,y)\varphi_{\omega}(\tau
-t,y)dydt,\qquad\tau\in(T_{0}(\omega),T_{\mathrm{med}}(\omega)],\qquad
\omega\in\mathbb{S}^{n-1}.
\]
with subsequent anti-differentiation in $\tau$ to obtain $\mathcal{R}%
f(\tau,\omega)$ on $\tau\in(T_{0}(\omega),T_{\mathrm{med}}(\omega)],$ and the
use of (\ref{E:thm1b}) to recover remaining values of $\mathcal{R}f.$

\subsection{Reconstruction from spatially reduced data\label{S:spat-reduced}}

The relative sparsity of the support of density $\varphi_{\omega}(t,y)$ can be
used to recover exactly the Radon projections of $f(x)$ from data supported on
$S\subseteq\Gamma,$ provided that support $\Omega_{0}$ of $f(x)$ lies on a
certain distance from $\Gamma\backslash S.$ Indeed, consider the distorted
representation $\delta^{\mathrm{distort}}(\tau,x)$ of the delta wave
$\delta(\tau-\omega\cdot x)$ arising if in equation (\ref{E:single-delta})
integration over $\Gamma$ is replaced by integration over $S$:%
\[
\delta^{\mathrm{distort}}(\tau,x)\equiv\int\limits_{T_{0}(\omega)}^{\tau}%
\int\limits_{S}\Phi_{n}(\tau-t,x-y)\varphi_{\omega}(t,y)dydt.
\]
Let us analyze the error $E(\tau,\omega,x)\equiv\delta(\tau-\omega\cdot
x)-\delta^{\mathrm{distort}}(\tau,x)$:%
\begin{equation}
E(\tau,\omega,x)=\int\limits_{T_{0}(\omega)}^{\tau}\int\limits_{\Gamma
\backslash S}\Phi_{n}(\tau-t,x-y)\varphi_{\omega}(t,y)dydt=\int\limits_{\Gamma
\backslash S}\int\limits_{\omega\cdot y}^{\tau}\Phi_{n}(\tau-t,x-y)\varphi
_{\omega}(t,y)dtdy, \label{E:def-error}%
\end{equation}
where we used (\ref{E:def-W}) and (\ref{E:small_support}). Due to
(\ref{E:small_support}) and to the finite speed of wave propagation, the
support of $E(\tau,\omega,x)$ is given by the formula
\begin{equation}
\text{support}(E(\tau,\omega,x))=\bigcup\limits_{y\in\Gamma\backslash
S,\quad\tau-\omega\cdot y>0}B(y,\tau-\omega\cdot y). \label{E:support}%
\end{equation}
In the above expression, $d(\tau,\omega,y)\equiv\tau-\omega\cdot y$ is a
signed distance from the point $y$ on $\Gamma\backslash S$ to the front
$L(\tau,\omega).$ This distance is positive for points that have been already
passed by the front; only these points contribute to (\ref{E:support}). If,
for a given $\Gamma$, $S$, and $\Omega_{0}$, parameters $\tau$ and $\omega$
are such that support$(E(\tau,\omega,x))$ does not intersect $\Omega_{0},$
then $\delta^{\mathrm{distort}}(\tau,x)$ can be used instead of $\delta
(\tau-\omega\cdot x)$ in the equation (\ref{E:findingJ}) without changing the
result. For such values of $\tau$ the Radon integral $\mathcal{R}f(\omega
,\tau)$ (or its derivative $\frac{\partial}{\partial\tau}\mathcal{R}%
f(\tau,\omega))$ are exactly reconstructed by the formulas%
\begin{align}
\mathcal{R}f(\tau,\omega)  &  =\int\limits_{0}^{\tau-T_{0}(\omega)}%
\int\limits_{S}G(t,y)\varphi_{\omega}(\tau-t,y)dydt,\label{E:reduced}\\
\frac{\partial}{\partial\tau}\mathcal{R}f(\tau,\omega)  &  =\int%
\limits_{0}^{\tau-T_{0}(\omega)}\int\limits_{S}g(t,y)\varphi_{\omega}%
(\tau-t,y)dydt. \label{E:reduced-g}%
\end{align}

Below we present several geometries where the set of pairs $(\tau,\omega)$ for
which (\ref{E:reduced}), (\ref{E:reduced-g}) hold is rich enough to
reconstruct, with the help of (\ref{E:Radon-symm}), the full set of the Radon
projections $\mathcal{R}f(\tau,\omega).$

Consider a 2D acquisition \textbf{geometry \#1} shown in
Figure~\ref{F:geometries}(a). The part of the boundary $\Gamma\backslash S$ is
given by a function $x_{2}=\gamma(x_{1})$ defined on the interval
$[-\alpha,\alpha].$ The endpoints of $\Gamma\backslash S$ are $a$ and $b$ with
coordinates $a=(-\alpha,\beta),$ $b=(\alpha,\beta).$ We assume additionally
that $\gamma(x_{1})\geq\beta$ on $[-\alpha,\alpha],$ and that the derivative
$\gamma^{\prime}(x_{1})$ is bounded. The open set $\Omega_{0}$ serving as the
support of $f(x)$ is bounded in $x_{2}$ by a line $x_{2}=\beta-\alpha.$ This
implies that the distance from $\Gamma\backslash S$ to any point in
$\Omega_{0}$ is strictly greater than $\alpha$. \begin{figure}[t]
\begin{center}%
\begin{tabular}
[c]{ccc}%
\includegraphics[scale = 0.8]{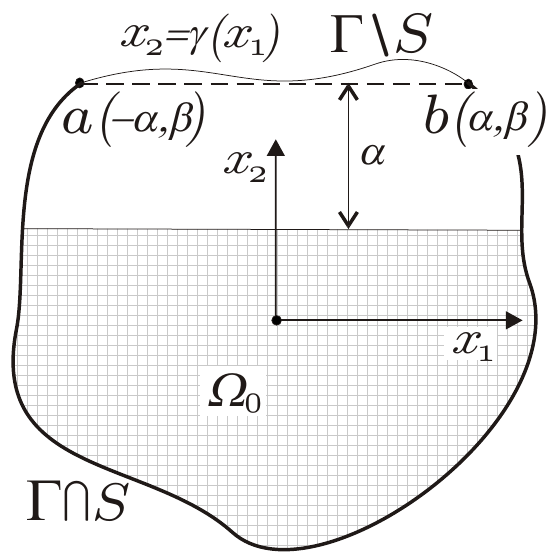}\phantom{a} & \phantom{a}
\includegraphics[scale = 0.8]{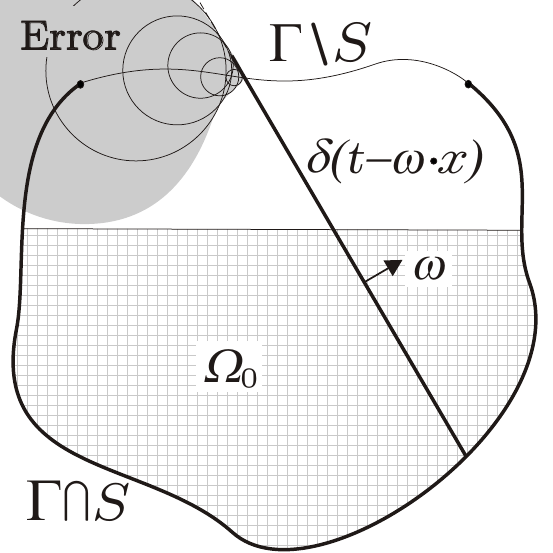}\phantom{a} & \phantom{a}
\includegraphics[scale = 0.8]{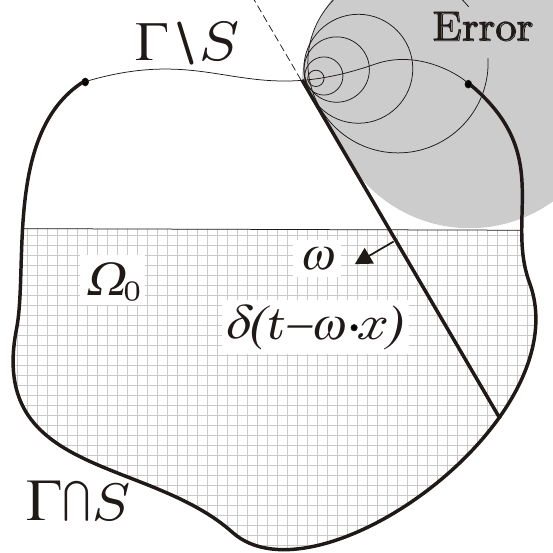} \phantom{a2}\\
(a) & (b) & (c)
\end{tabular}
\end{center}
\par
\vspace{-3mm} \caption{(a) acquisition geometry \#1; (b),(c) support of the
error $E(\tau,\omega,x)$ (shown as gray areas) in the representation of
$\delta(t-\omega\cdot x)$ for different directions $\omega$. }%
\label{F:geometries}%
\end{figure}

\begin{proposition}
\label{T:intervals}For the acquisition \textbf{geometry \#1}, formula
(\ref{E:reduced}) holds for all $\omega\equiv(\omega_{1},\omega_{2}%
)\neq(0,-1)$ with $\tau$ lying within the following intervals:%
\begin{equation}
\tau\in\left\{
\begin{array}
[c]{cc}%
(T_{0}(\omega),\omega\cdot a+\alpha], & \omega_{1}\geq0,\\
(T_{0}(\omega),\omega\cdot b+\alpha], & \omega_{1}<0.
\end{array}
\right.  \label{E:intervals}%
\end{equation}

\end{proposition}

\begin{proof}
\noindent\emph{Case 1}. Consider the Heaviside wave propagating in the
direction $\omega$ with $\omega_{1}\geq0,$ $\omega_{2}\geq0.$ For $\tau
\in(T_{0}(\omega),\omega\cdot a]$, support$(E(\tau,\omega,x))=\varnothing.$
For $\tau>\omega\cdot a,$ the largest $d(\tau,\omega,y)$ over points
$y\in\Gamma\backslash S$ equals to $\tau-\omega\cdot a$, since
\[
d(\tau,\omega,y)=\tau-\omega_{1}y_{1}-\omega_{2}\gamma(y_{1})\leq\tau
-\omega_{1}y_{1}-\omega_{2}\beta\leq\tau-\omega_{1}(-\alpha)-\omega_{2}%
\beta=d(\tau,\omega,a).
\]
Moreover, for $\tau\leq\omega\cdot a+\alpha,$ $d(\tau,\omega,a)\leq\alpha,$
and thus $d(\tau,\omega,y)\leq$ $\alpha$ for all $y$ in (\ref{E:support}).
Therefore, support$(E(\tau,\omega,x))$ does not intersect $\Omega_{0},$ and
formulas (\ref{E:reduced}), (\ref{E:reduced-g}) hold.

\noindent\emph{Case }2. Consider the wave with $\omega=(\omega_{1},\omega
_{2})$ with $\omega_{1}\geq0,$ $-1<\omega_{2}<0.$ Given $\tau\leq\omega\cdot
a+\alpha,$ with $e_{2}=(0,1)$ one obtains
\[
\mathrm{dist}(y,\Omega_{0})\geq\alpha+(y-a)\cdot e_{2}\geq\tau-\omega\cdot
a+(y-a)\cdot e_{2}=d(\tau,\omega,y)+(y-a)\cdot(\omega+e_{2}).
\]
Notice that vector $y-a$ lies in the first quadrant (both coordinates are
non-negative). The same is true for vector $\omega+e_{2}.$ Therefore,
$(y-a)\cdot(\omega+e_{2})\geq0$ and $d(\tau,\omega,y)\leq$dist(y,$\Omega
_{0}).$ Thus, support$(E(\tau,\omega,x))$ does not intersect $\Omega_{0},$ and
formulas (\ref{E:reduced}), (\ref{E:reduced-g}) hold.

\noindent\emph{Case 3}. When $\omega_{1}<0$, one can argue similarly to cases
1 and 2, with the vector $a$ replaced by $b.$
\end{proof}

With a little trigonometry, conditions (\ref{E:intervals}) can be represented
more succinctly. Express $b$ in the form $b=|b|(\sin\mu,\cos\mu);$ then
$\ \mu\in(0,\pi)$ is the\ angle between $e_{2}$ and $b$. Further,
$\alpha=|b|\sin\mu,$ $\beta=|b|\cos\mu.$ Now let us represent $(-\omega)$ as
$-\omega=(-\omega_{1},\cos\nu),$ with $\nu\in\lbrack0,\pi].$ Parameter $\nu$
is the (positive) angle between $e_{2}$ and $(-\omega)$. Then one can check
that (\ref{E:intervals}) is equivalent to
\begin{equation}
\tau\in(T_{0}(\omega),|a|(-\cos(\mu-\nu)+\sin\mu)],\qquad\nu\in(0,\pi].
\label{E:new-intervals}%
\end{equation}
As we will show shortly, the set of values $\mathcal{R}f(\tau,\omega)$ \ for
$(\tau,\omega)$ satisfying (\ref{E:new-intervals}) contains enough information
to reconstruct all values of $\mathcal{R}f(\tau,\omega)$. Moreover, this set
can be further reduced. For each $\omega\in\mathbb{S}^{1}$ (except
$\omega=(0,-1))$ consider the following intervals of values of $\tau$:%
\begin{equation}
\tau\in\left\{
\begin{array}
[c]{cc}%
(T_{0}(\omega),|a|(-\cos(\mu-\nu)+\sin\mu)], & \nu\in(0,\pi/2],\\
(T_{0}(\omega),|a|(-\cos(\mu+\nu)-\sin\mu)], & \nu\in\lbrack\pi/2,\pi].
\end{array}
\right.  \label{E:short}%
\end{equation}
Let us compare the upper bound of the second interval in (\ref{E:short}) with
that in (\ref{E:new-intervals}):%
\[
(-\cos(\mu-\nu)+\sin\mu)-(-\cos(\mu+\nu)-\sin\mu)=2(1-\sin\nu)\sin\mu
\geq0,\quad\forall\mu,\nu\in(0,\pi].
\]
Therefore, intervals in (\ref{E:short}) are contained in
(\ref{E:new-intervals}), or, equivalently, in (\ref{E:intervals}).

\begin{proposition}
\label{T:geom1} For the acquisition \textbf{geometry \#1}, formulas
(\ref{E:reduced}), (\ref{E:reduced-g}) hold for all $\omega\neq(0,-1)$ and
$\tau$ lying within the intervals (\ref{E:short}). Moreover, by finding
$\mathcal{R}f(\tau,\omega)$ for these values of~$\omega$ and~$\tau,$ one
obtains enough information to reconstruct $\mathcal{R}f(\tau,\omega)$ for all
values of $\omega\in\mathbb{S}^{1},$ $\tau\in\mathcal{T}(\omega),$
using~(\ref{E:Radon-symm}).
\end{proposition}

\begin{proof}
The first part of the statement follows from Proposition \ref{T:intervals}. To
prove the second part, consider two opposite directions of the wave $\omega$
and $\omega^{\prime}=-\omega.$ Then, there are (non-negative) angles $\nu$ and
$\nu^{\prime}$ such that $-\omega_{2}=\cos\nu,$ $-\omega_{2}^{\prime}=\cos
\nu^{\prime},$ with $\nu^{\prime}=\pi-\nu.$ Without loss of generality, assume
that $\nu\in(0,\pi/2];$ then $\nu^{\prime}\in\lbrack\pi/2,\pi).$ The first
line in (\ref{E:short}) allows one to reconstruct values of $\mathcal{R}%
f(\tau,\omega)$ for
\begin{equation}
\tau\in(T_{0}(\omega),|a|(-\cos(\mu-\nu)+\sin\mu)]. \label{E:int1}%
\end{equation}
The second line in (\ref{E:short}) yields values of $\mathcal{R}f(\tau
^{\prime},\omega^{\prime})$ for $\tau^{\prime}\in(T_{0}(\omega^{\prime
}),|a|(-\cos(\mu+\nu^{\prime})+\sin\mu)].$ Due to (\ref{E:Radon-symm}),
$\mathcal{R}f(\tau^{\prime},\omega^{\prime})=\mathcal{R}f(-\tau^{\prime
},-\omega^{\prime})=\mathcal{R}f(\tau,\omega)$ with
\begin{equation}
\tau\in(-|a|(-\cos(\mu+\nu^{\prime})+\sin\mu),-T_{0}(\omega^{\prime
}))=(|a|(-\cos(\mu-\nu)-\sin\mu),T_{1}(\omega)). \label{E:int2}%
\end{equation}
Combining the data obtained for intervals (\ref{E:int1}) and (\ref{E:int2})
yields $\mathcal{R}f(\tau,\omega)$ for all $\tau\in(T_{0}(\omega),T_{1}%
(\omega))\equiv\mathcal{T}(\omega).$ For the remaining case $\nu=\pi$ it is
enough to consider the second line in (\ref{E:short}), since lines with
$\tau\geq|a|(\cos(\mu-\sin\mu)$ do not intersect $\Omega_{0}.$
\end{proof}

As a particular case of \textbf{geometry \#1}, consider \textbf{open circular
geometry}, where $\Gamma$ is a unit circle centered at the origin. Then, for a
given parameter $\mu\in(0,\pi/2),$ $a=(-\sin\mu,\cos\mu)$ and $b=(\sin\mu
,\cos\mu);$ $\Gamma\backslash S$ is the upper arc of the circle between $a$
and $b.$ Then $\alpha=\sin\mu,$ and $\Omega_{0}$ is the part of the open unit
disk lying under the line $x_{2}=\cos\mu-\sin\mu.$ As before, $-\omega
=(-\omega_{1},\cos\nu),$ with $\nu\in\lbrack0,\pi].$

\begin{corollary}
\label{T:geomcirc}For the \textbf{open circular geometry}, formulas
(\ref{E:reduced}), (\ref{E:reduced-g}) hold for all $\omega\neq(0,-1)$ and
$\tau$ lying within the intervals%
\begin{equation}
\tau\in\left\{
\begin{array}
[c]{cc}%
(-1,-\cos(\mu-\nu)+\sin\mu)], & \nu\in(0,\pi/2],\\
(-1,-\cos(\mu+\nu)-\sin\mu)], & \nu\in\lbrack\pi/2,\pi].
\end{array}
\right.  \label{E:circle-int}%
\end{equation}
Values of $\mathcal{R}f(\tau,\omega)$ for all values of $\omega\in
\mathbb{S}^{1},$ $\tau\in(-1,1),$ can be obtained using~(\ref{E:Radon-symm}).
\end{corollary}

We notice that for small values of parameter $\mu\in(0,\pi/2),$ the part of
the boundary $\Gamma\backslash S$ where measurements are not taken, is very
small. For values of $\mu$ close to $\pi/2,$ region $\Omega_{0}$ is too small
to be of practical interest. Values from the middle of the interval,
hopefully, represent a useful compromise. For example, for $\mu=\pi/4,$ region
$\Omega_{0}$ is the lower half of the open unit disk, and the opening
$\Gamma\backslash S$ in $\Gamma$ is a $90^{\circ}$ arc. The maximum length of
intervals in (\ref{E:circle-int})\ is attained for $\nu=3\pi/4;$ it is equal
to $2-1/\sqrt{2}.$ Although the diameter of $\Omega_{0}$ is $2,$ only the data
in a reduced temporal range of $[0,2-1/\sqrt{2}]\approx\lbrack0,1.293]$ are
utilized. Thus, theoretically exact reconstruction in the case of \textbf{open
circular geometry} is achieved from the data reduced both spatially and temporally.

The same line of reasoning can be extended to the 3D case. We do not formulate
a general statement, and consider instead a particular case of an \textbf{open
spherical geometry} defined as follows. Surface $\Gamma$ is a unit sphere
centered at the origin, $\Gamma\backslash S$ is the part of the unit sphere
lying above the horizontal plane $x_{3}=\cos\mu,$ with $\mu\in(0,\pi/2),$ and
$\Omega_{0}$ is the part of the open unit ball lying under the horizontal
plane $x_{3}=\cos\mu-\sin\mu.$ Let us represent $-\omega$ in the form
$-\omega=(-\omega_{1},-\omega_{2},\cos\nu)$\ with $\nu\in\lbrack0,\pi].$ As
before, $\nu$ is the (positive) angle between $e_{3}$ and $(-\omega)$.

\begin{proposition}
\label{T:geomsphere}For the \textbf{open spherical geometry}, formulas
(\ref{E:reduced}), (\ref{E:reduced-g}) hold for all $\omega\neq(0,0,-1)$ and
$\tau$ lying within the intervals%
\begin{equation}
\tau\in\left\{
\begin{array}
[c]{cc}%
(-1,-\cos(\mu-\nu)+\sin\mu)], & \nu\in(0,\pi/2],\\
(-1,-\cos(\mu+\nu)-\sin\mu)], & \nu\in\lbrack\pi/2,\pi].
\end{array}
\right.  \label{E:int-sphere}%
\end{equation}
Values of $\mathcal{R}f(\tau,\omega)$ for all values of $\omega\in
\mathbb{S}^{2},$ $\tau\in(-1,1),$ can be obtained using~(\ref{E:Radon-symm}).
\end{proposition}

Due to the rotational symmetry with respect to the vertical axis (spanned by
vector $e_{3}=(0,0,1)),$ in order to prove this proposition it is enough to
consider directions $\omega$ in the form $\omega=(\omega_{1},0,\omega_{3}).$
The proof is completed by considerations very similar to those in Propositions
\ref{T:intervals} and \ref{T:geom1}, and in Corollary \ref{T:geomcirc}.
Similarly to the circular case, only values of $G(t,y)$ in the reduced
temporal range of $[0,2-1/\sqrt{2}]$ are used in Proposition
\ref{T:geomsphere}, thus delivering theoretically exact reconstruction from
both spatially and temporally reduced data.

The explicit reconstruction techniques presented in this section are somewhat
suboptimal in that the "visible" regions corresponding to the acquisition
surfaces $S$ in our geometries are larger than our $\Omega_{0}$'s$.$ On the
other hand, these are currently \emph{the only known results}\textbf{ }for
explicit PAT/TAT reconstruction from bounded open surfaces.

\subsection{Remarks on the practical aspects of the method}

The method we propose is explicit only if the density $\varphi_{\omega}(t,y)$
is known. For a spherical or circular region $\Omega^{-}$ this function can be
obtained by exploiting the spherical symmetry and expanding functions in
spherical or circular harmonics. We analyze this case in detail in the next
section. For more general acquisition surfaces, the set of single layer
densities $\varphi_{\omega}(t,y)$ for each $\omega\in\mathbb{S}^{n-1}$ should
be pre-computed once and stored. This problem is not necessarily easy.
However, as was mentioned in Section \ref{S:repres}, it is closely related to
the problem of the time domain acoustic surface scattering. There exists a
significant body of work concerning both the theory and the methods of
efficient numerical computation of such functions, see \cite{Sayas} and
references therein.

Another practical concern is the number of operations required to realize this
technique, assuming that the densities are known. For each fixed $\omega
\in\mathbb{S}^{n-1}$, the density is a function of time and boundary
coordinates, making them functions of $n$ dimensions in the space
$\mathbb{R}^{n}.$ For each $\tau$ an $n$-dimensional integral should be
computed. If discretization in all variables is done using $\mathcal{O}(m)$
points, each value of $\mathcal{R}f(\tau,\omega)$ will require $\mathcal{O}%
(m^{n})$ floating point operations (flops). Since full Radon data would have
$\mathcal{O(}m^{n})$ data points, the total computational cost would be
$\mathcal{O(}m^{2n}),$ or $\mathcal{O(}m^{4})$ in 2D and $\mathcal{O(}m^{6})$
in 3D. For comparison, a backprojection step of a filtration/backprojection
algorithm requires $\mathcal{O(}m^{3})$ flops in 2D and $\mathcal{O(}m^{5})$
flops in 3D. In other words, a straightforward implementation of formulas
(\ref{E:thm1a}), (\ref{E:thm1b}) would lead to a slower algorithm than
filtration/backprojection techniques.

However, even for a general acquisition surface $\Gamma,$ the present method
can be accelerated. The time-consuming computation of integrals can be
reformulated as a standard convolution in time (see \ equation
(\ref{E:convolution})); further, the integration order can be interchanged
yielding
\begin{equation}
\mathcal{R}f(\tau,\omega)=\int\limits_{S}\int\limits_{\mathbb{R}}%
G(t,y)\varphi_{\omega}(\tau-t,y)dtdy,\qquad\tau\in(T_{0}(\omega
),T_{\mathrm{med}}(\omega)],\qquad\omega\in\mathbb{S}^{n-1}. \label{E:convo}%
\end{equation}
The inner integral in (\ref{E:convo}) is a convolution that can be computed
using Fast Fourier Transform (FFT) techniques, separately for each point $y\in
S$ and each vector $\omega\in\mathbb{S}^{n-1}$. This way, for each $\omega,$
the values for all $\tau\in(T_{0}(\omega),T_{\mathrm{med}}(\omega))$ are found
at once, at the price of $\mathcal{O(}m^{n}\log m)$ flops. This needs to be
repeated for all $\omega\in\mathbb{S}^{n-1},$ resulting in $\mathcal{O(}%
m^{2n-1}\log m)$ flops algorithm. Such an accelerated method would require
$\mathcal{O(}m^{3}\log m)$ flops in 2D and $\mathcal{O(}m^{5}\log m)$ flops in
3D, which is slower, but not catastrophically, than filtration/backprojection routines.

If the acquisition surface have rotational symmetries, the computation can be
further accelerated. The algorithms for spherical and circular acquisition
surfaces presented in the next sections are significantly faster.

\section{Circular and spherical geometries\label{S:Spheres}}

In this section we find explicit expressions for the density of single layer
potentials for circular and spherical acquisition geometries, and develop
efficient numerical algorithms that exploit the arising rotational symmetries.

We will utilize the forward and inverse Fourier transforms $\mathcal{F}$ and
$\mathcal{F}^{-1}$ on $\mathbb{R}$ defined as follows:%
\begin{equation}
\left(  \mathcal{F}h\right)  (\rho)=\int\limits_{\mathbb{R}}h(t)e^{i\rho
t}dt,\qquad\left(  \mathcal{F}^{-1}\hat{h}\right)  (t)=\frac{1}{2\pi}%
\int\limits_{\mathbb{R}}\hat{h}(\rho)e^{-i\rho t}d\rho,\qquad t,\rho
\in\mathbb{R}.\nonumber
\end{equation}
The forward transform will be applied with respect to the time variable. In
particular, let us Fourier transform equation (\ref{E:Gr-conv}). This yields%

\begin{equation}
\hat{G}(\rho,y)\equiv\left(  \mathcal{F}G\right)  (\rho,y)=\int\limits_{\Omega
}f(x)\left[  \mathcal{F}\Phi_{n}\right]  (\rho,y-x)dx=\int\limits_{\Omega
}f(x)\hat{\Phi}_{n}(\rho,y-x)dx, \label{E:Fourier-convolution}%
\end{equation}
where $\hat{G}(\rho,y)$ is the Fourier transform (in time) of the data
$G(t,y)$, and $\hat{\Phi}_{n}(\rho,x)\equiv\left[  \mathcal{F}\Phi_{n}\right]
(\rho,x)$ is the Fourier transform of the fundamental solution $\Phi_{n}$ of
the wave equation. It is well known that%
\begin{equation}
\hat{\Phi}_{2}(\rho,x)=\frac{i}{4}H_{0}^{(1)}(\rho|x|),\qquad\hat{\Phi}%
_{3}(\rho,x)=\frac{i\rho}{4\pi}h_{0}^{(1)}(\rho|x|),
\label{E:Fourier-Green-3d}%
\end{equation}
where $H_{0}^{(1)}$ and $h_{0}^{(1)}$ are, respectively, the regular and
spherical Hankel functions of order~0. Each of $\hat{\Phi}_{n}(\rho,x),$
$n=2,3,$ is also the fundamental solution of the Helmholtz equation in
$\mathbb{R}^{n},$ satisfying radiation condition at infinity\cite{Colton}.
Importantly, with the bar denoting complex conjugation,%
\[
\overline{\hat{\Phi}_{n}(\rho,x)}=\hat{\Phi}_{n}(-\rho,x),\qquad\rho
\in\mathbb{R}.
\]
Instead of surface scattering problems, in this section we use harmonic
analysis in order to derive single layer representations for delta-waves. We
start by representing a family of smooth approximations to the delta wave
$\delta(\tau-\omega\cdot x),$ and then pass to the limit. Let us consider a
$C_{0}^{\infty}(\mathbb{R})$ function$~\eta,$ compactly supported on the
interval $(-1,1),$ such that $\int_{\mathbb{R}}\eta(t)dt=\int_{\mathbb{-}%
1}^{1}\eta(t)dt=1.$ Now we can define a delta-approximating family of
functions $\eta_{\varepsilon}(t)$ as follows%
\[
\eta_{\varepsilon}(t)\equiv\frac{1}{\varepsilon}\eta\left(  \frac
{t}{\varepsilon}-1\right)  .
\]
For a fixed $\varepsilon>0,$ function $\eta_{\varepsilon}(t)$ is defined on
$\mathbb{R}$ and finitely supported on $(0,2\varepsilon).$ Moreover,
\[
\int\limits_{\mathbb{R}}\eta_{\varepsilon}(t)dt=1,\qquad\text{and}\qquad\text{
}\eta_{\varepsilon}(t)\underset{\varepsilon\rightarrow0}{\rightarrow}%
\delta(t).
\]
Denote the Fourier transform of $\eta(t)$ by $\hat{\eta}_{\varepsilon}(t).$
Then $\eta_{\varepsilon}(t)=\frac{1}{2\pi}\int\limits_{\mathbb{R}}\hat{\eta
}_{\varepsilon}(\rho)e^{-i\rho t}d\rho,$ and the smooth plane wave
$\eta_{\varepsilon}(\tau-\omega\cdot x)$ on $\mathbb{R\times R}^{n}$ can be
expressed as follows:%
\begin{equation}
\eta_{\varepsilon}(\tau-\omega\cdot x)=\frac{1}{2\pi}\int\limits_{\mathbb{R}%
}\hat{\eta}_{\varepsilon}(\rho)e^{-i\rho(\tau-\omega\cdot x)}d\rho=\frac
{1}{2\pi}\int\limits_{\mathbb{R}}\hat{\eta}_{\varepsilon}(\rho)e^{-i\rho\tau
}e^{i\rho\omega\cdot x}d\rho. \label{E:smooth-wave}%
\end{equation}
In the limit $\varepsilon\rightarrow0$, plane waves $\eta_{\varepsilon}\left(
\tau-\omega\cdot x\right)  $ converge to $\delta\left(  \tau-\omega\cdot
x\right)  $ in the space of distributions.

\subsection{2D: circular and open circular geometries}

\subsubsection{Expression for the density}

Let us represent a plane wave $e^{i\rho\omega\cdot x}$ by a time-harmonic
single layer potential. \ \ We express vectors $x,\hat{y},$ and $\omega$ in
polar coordinates as $x=r(\cos\theta,\sin\theta),$ \ $\hat{y}=(\cos\psi
,\sin\psi),$ and $\omega=(\cos\varpi,\sin\varpi).$ The expansion of the plane
wave $e^{i\rho\omega\cdot x}$ in \ Fourier series with respect to angle
$\theta$ is given by the Jacobi-Anger formula \cite{Colton}:%
\begin{equation}
e^{i\rho\omega\cdot x}=\sum\limits_{k=-\infty}^{\infty}i^{|k|}J_{|k|}(\rho
r)e^{ik\varpi}e^{-ik\theta}. \label{E:Jacobi}%
\end{equation}
The fundamental solution $\hat{\Phi}_{2}(\rho,x)$ can be expanded using the
addition theorem for the Hankel function $H_{0}^{(1)}$ \cite{Colton}:%
\begin{equation}
H_{0}^{(1)}(\rho|\hat{y}-x|)=\sum\limits_{k=-\infty}^{\infty}H_{|k|}%
^{(1)}(\rho)J_{|k|}(\rho r)e^{ik(\psi-\theta)},\quad r<1.
\label{E:addition-2D}%
\end{equation}
The above formula allows us to express circular waves $J_{|k|}(\rho
r)e^{-ik\theta}$ by the single layer potentials:%
\begin{equation}
\frac{2}{\pi i}\int\limits_{\mathbb{S}^{1}}\frac{e^{-ik\psi}}{H_{|k|}%
^{(1)}(\rho)}\hat{\Phi}_{2}(\rho,x-\hat{y})d\hat{y}=J_{|k|}(\rho
r)e^{-ik\theta},\quad\hat{y}(\psi)=(\cos\psi,\sin\psi),\quad|x|<1,\quad
\rho\geq0, \label{E:single-2D}%
\end{equation}
By combining (\ref{E:single-2D})\ and (\ref{E:Jacobi})\ one obtains the
following representation for the plane wave:%
\begin{equation}
\int\limits_{\mathbb{S}^{1}}\left[  \sum\limits_{k=-\infty}^{\infty}%
\frac{2i^{|k|}e^{ik(\varpi-\psi)}}{\pi iH_{|k|}^{(1)}(\rho)}\right]  \hat
{\Phi}_{2}(\rho,x-\hat{y})d\hat{y}=e^{i\rho\omega\cdot x},\quad|x|<1,\quad
\rho\geq0. \label{E:2D-ugly}%
\end{equation}
Define function $\hat{\varphi}_{\omega}(\rho,\hat{y})$ through the expression
in brackets in (\ref{E:2D-ugly}) and its conjugate (for $\rho<0)$:%
\begin{equation}
\hat{\varphi}_{\omega}(\rho,\hat{y}(\psi))\equiv\frac{2}{\pi i}\sum
\limits_{k=-\infty}^{\infty}\frac{i^{|k|}e^{ik(\varpi-\psi)}}{H_{|k|}%
^{(1)}(\rho)},\quad\rho\geq0;\qquad\hat{\varphi}_{\omega}(-\rho,\hat{y}%
)\equiv\overline{\hat{\varphi}_{\omega}(\rho,\hat{y})}. \label{E:def-dens-2D}%
\end{equation}
With so defined $\hat{\varphi}_{\omega}(\rho,\hat{y})$ the following formula
holds for all real $\rho$:
\begin{equation}
\int\limits_{\mathbb{S}^{1}}\hat{\varphi}_{\omega}(\rho,\hat{y})\hat{\Phi}%
_{2}(\rho,x-\hat{y})d\hat{y}=e^{i\rho\omega\cdot x},\quad|x|<1,\quad\rho
\in\mathbb{R}. \label{E:gen-plane-2D}%
\end{equation}
By multiplying (\ref{E:gen-plane-2D}) with $\frac{1}{2\pi}\hat{\eta
}_{\varepsilon}(\rho)e^{-i\rho\tau},$ integrating over $\mathbb{R}$, and using
(\ref{E:smooth-wave}), one obtains%
\[
\int\limits_{\mathbb{S}^{1}}\left(  \frac{1}{2\pi}\int\limits_{\mathbb{R}%
}\left[  \hat{\varphi}_{\omega}(\rho,\hat{y})\hat{\eta}_{\varepsilon}%
(\rho)\right]  \hat{\Phi}_{2}(\rho,x-\hat{y})e^{-i\rho\tau}d\rho\right)
d\hat{y}=\eta_{\varepsilon}\left(  \tau-\omega\cdot x\right)  ,\quad|x|<1.
\]
Since $\hat{\eta}_{\varepsilon}(\rho)$ decays at infinity faster than any
rational function of $\rho,$ the product $\hat{\varphi}_{\omega}(\rho,\hat
{y})\hat{\eta}_{\varepsilon}(\rho)$ can be Fourier transformed. Define
$\varphi_{\omega,\varepsilon}(\tau,\hat{y})$ on $\mathbb{R}\times
\mathbb{S}^{1}$ as follows%
\[
\varphi_{\omega,\varepsilon}(\tau,\hat{y})=\mathcal{F}^{-1}(\hat{\eta
}_{\varepsilon}(\rho)\hat{\varphi}_{\omega}(\rho,\hat{y}))(\tau,\hat{y}),
\]
Then, plane wave $\eta_{\varepsilon}\left(  \tau-\omega\cdot x\right)  $ is
represented by a single layer potential with density $\varphi_{\omega
,\varepsilon}(\tau,\hat{y})$
\[
\eta_{\varepsilon}\left(  \tau-\omega\cdot x\right)  =\int\limits_{\mathbb{S}%
^{1}}\left(  \int\limits_{\mathbb{R}}\varphi_{\omega,\varepsilon}(t,\hat
{y})\Phi_{2}(\tau-t,x-\hat{y})dt\right)  d\hat{y}.
\]
For each fixed $\varepsilon,$ wave $\eta_{\varepsilon}\left(  \tau-\omega\cdot
x\right)  $ is in the form (\ref{E:planewave}) and satisfies (\ref{E:finitesp}%
). Therefore, due to Proposition \ref{T:exist}, density $\varphi
_{\omega,\varepsilon}(t,\hat{y})$ vanishes on $W\cap(\mathbb{R\times S}^{1}),$
where $W$ is defined by (\ref{E:def-W}), with $\mathcal{T}(\omega)=(-1,1).$

Now, let us take the limit $\varepsilon\rightarrow0.$ In this limit,
$\eta_{\varepsilon}\left(  \tau-\omega\cdot x\right)  \rightarrow\delta\left(
\tau-\omega\cdot x\right)  .$ On the other hand, functions $\varphi
_{\omega,\varepsilon}(t,\hat{y})$ converge to the distribution $\varphi
_{\omega}(t,\hat{y})$ defined by its Fourier transform\ (\ref{E:def-dens-2D}).
Formula (\ref{E:single-delta}) holds for this density $\varphi_{\omega}%
(t,\hat{y})$. Importantly, the limit distribution $\varphi_{\omega}(t,\hat
{y})$ inherits the sparse support of functions $\varphi_{\omega,\varepsilon
}(t,\hat{y}),$ i.e. $\varphi_{\omega}(t,\hat{y})$ vanishes on $W\cap
(\mathbb{R\times S}^{1}).$ Moreover, all considerations of Section
\ref{S:spat-reduced} apply in the present case, allowing one to reconstruct
Radon projections from the data reduced both spatially and temporally, using
formula (\ref{E:reduced}) or (\ref{E:reduced-g}).

\subsubsection{Fast 2D algorithm and simulations\label{S:fast2D}}

In this section we develop a fast algorithm for the spatially reduced circular
geometry. The algorithm reconstructs Radon projections $\mathcal{R}%
f(\tau,\omega)$ from thermoacoustic data $g(t,\hat{y})$ measured on
$[0,2-1/\sqrt{2}]\times S.$ Following Proposition \ref{T:geomsphere}, our
technique involves three stages (I) reconstruction of $\frac{\partial
}{\partial\tau}\mathcal{R}f(\tau,\omega)$ using formula (\ref{E:reduced-g})
for values of $\tau$ given by inequalities (\ref{E:circle-int}) (II)
anti-differentiation in $\tau$ to obtain $\mathcal{R}f(\tau,\omega)$ on the
same intervals (III)\ reconstruction the rest of the projections
using~(\ref{E:Radon-symm}).

The latter two stages are almost trivial. The first stage of the algorithm can
be viewed as a convolution performed over the cylinder $\mathbb{R\times S}%
^{1}.$ Indeed, let us extend $g(t,\hat{y})$ by zero to all $\hat{y}\in
\Gamma\backslash S,$ and to all $t\notin$ $[0,2-1/\sqrt{2}].$ We will denote
the extended data by $\widetilde{g}(t,\hat{y});$ it is defined on
$\mathbb{R\times S}^{1}.$ Then formula (\ref{E:reduced-g})\ can be re-written
in the form:
\begin{equation}
\widetilde{\frac{\partial}{\partial\tau}\mathcal{R}f}(\tau,\omega
)=\int\limits_{\mathbb{R}}\int\limits_{\mathbb{S}^{1}}\widetilde{g}(t,\hat
{y})\varphi_{\omega}(\tau-t,\hat{y})d\hat{y}dt, \label{E:circ-conv}%
\end{equation}
where function $\widetilde{\frac{\partial}{\partial\tau}\mathcal{R}f}%
(\tau,\omega)$ is equal to 0 for all $\tau\leq-1;$ it coincides with
$\frac{\partial}{\partial\tau}\mathcal{R}f(\tau,\omega)$ on intervals
(\ref{E:circle-int}), and, due to the finite support of $t \mapsto
\widetilde{g}(t,\hat{y})$, it decays in the limit $t\rightarrow+\infty$ with
the rate of decay determined by that of $\varphi_{\omega}(t,\hat{y})$. Due to
the rotational symmetry of the integration surface, formula (\ref{E:circ-conv}%
) represents a convolution over $\mathbb{R\times S}^{1}.$ An efficient and
natural way of computing such a convolution is by using the Fourier
transform/Fourier series techniques. In fact, the expansion of the convolution
kernel $\varphi_{\omega}(t,\hat{y})$ in Fourier components is already given by
equation (\ref{E:def-dens-2D}). The remaining details are as follows.

By computing the Fourier transform of equation (\ref{E:circ-conv}) in $\tau,$
and expanding the data in Fourier series in $\psi$ (where $\hat{y}=\hat
{y}(\psi)$) one obtains%
\[
\widehat{\widetilde{\frac{\partial}{\partial\tau}\mathcal{R}f}}(\rho
,\omega)=\int\limits_{\mathbb{S}^{1}}\widehat{\widetilde{g}}(\rho,\hat{y}%
)\hat{\varphi}_{\omega}(\rho,\hat{y})d\hat{y},
\]
where $\widehat{\cdot}$ denotes the Fourier transform in $\tau.$ Further, for
$\rho\geq0,$ using the definition of $\hat{\varphi}_{\omega}$ (see
(\ref{E:def-dens-2D})) we find that%
\begin{equation}
\widehat{\widetilde{\frac{\partial}{\partial\tau}\mathcal{R}f}}(\rho
,\omega)=\sum\limits_{k=-\infty}^{\infty}\frac{4}{i}\frac{i^{|k|}%
\widehat{\widetilde{g}}_{k}(\rho)}{H_{|k|}^{(1)}(\rho)}e^{ik\varpi},\quad
\rho\geq0, \label{E:2d-coeffs}%
\end{equation}
where $\widehat{\widetilde{g}}_{k}(\rho)$ are the expansion coefficients of
$\widehat{\widetilde{g}}(\rho,\hat{y}(\psi))$ in Fourier series:%
\begin{equation}
\widehat{\widetilde{g}}_{k}(\rho)=\frac{1}{2\pi}\int\limits_{0}^{2\pi
}\widehat{\widetilde{g}}(\rho,\hat{y}(\psi))e^{-ik\psi}d\psi.
\label{E:Fourier-ser}%
\end{equation}
Formulas (\ref{E:2d-coeffs}) and (\ref{E:Fourier-ser}) lead to the following
algorithm for the open circular geometry. Let us define uniform computational
grids in $\tau,$ $t,$ $\rho,$ $\psi,$ and $\varpi,$ and assume that extended
data $\widetilde{g}(t,\hat{y}(\psi))$ is sampled on the product grid in $t$
and $\psi.$ We then

\noindent\textbf{1.} Expand $\widetilde{g}(t,\hat{y})$ in the Fourier series,
and Fourier-transform the result to obtain $\widehat{\widetilde{g}}_{k}(\rho)$
for each grid value of $\rho\geq0;$

\noindent\textbf{2.} For each grid value of $\rho\geq0,$ compute coefficients
$b_{k}(\rho)\equiv\frac{4}{i}\frac{i^{|k|}}{H_{|k|}^{(1)}(\rho)}%
\widehat{\widetilde{g}}_{k}(\rho)$ and extend them to negative $\rho$'s by
complex conjugation;

\noindent\textbf{3.} For each grid value of $\rho,$ sum up the series
$\sum\limits_{k}b_{k}(\rho)e^{ik\varpi}$, and apply the inverse Fourier
transform in $\rho,$ thus finding $\widetilde{\frac{\partial}{\partial\tau
}\mathcal{R}f}(\tau,\omega);$

\noindent\textbf{4.} Anti-differentiate $\widetilde{\frac{\partial}%
{\partial\tau}\mathcal{R}f}(\tau,\omega)$ to find $\widetilde{\mathcal{R}%
f}(\tau,\omega)$ within the intervals (\ref{E:circle-int});

\noindent\textbf{5.} Compute $\mathcal{R}f(\tau,\omega)$ by extracting the
correct values of $\widetilde{\mathcal{R}f}(\tau,\omega)$ within the intervals
(\ref{E:circle-int}), and finding remaining values using (\ref{E:Radon-symm}).

For the full circle acquisition in the reduced temporal range $t\in
\lbrack0,1]$, step 5 is replaced by step

\noindent\textbf{5}$^{\ast}$\textbf{.} Compute $\mathcal{R}f(\tau,\omega)$ by
extracting the correct values of $\widetilde{\mathcal{R}f}(\tau,\omega)$
within the interval $\tau\in\lbrack-1,0]$ for all $\omega\in\mathbb{S}^{1}$,
and finding remaining values using (\ref{E:Radon-symm}).

The reader may have already noticed that the first three steps of the above
algorithm coincide with the fast \emph{modified Norton algorithm} for a
circular (2D) acquisition geometry \cite{Kun-cyl}. The latter reference
contains the details of a fast implementation that requires only
$\mathcal{O}(m^{2}\log m)$ flops for an $m\times m$ image. In order to obtain
accurate reconstructions, we had to introduce a small modification in this
algorithm. Namely, for the several first terms of Fourier series
($k=-4,3,...,3,4)$ we use a significantly over-sampled (by a $\emph{factor}$
of $32)$ uniform grid in $\rho.$ This is a consequence of non-analytic
behavior of the terms with small $|k|$ near $\rho=0$ in the series giving by
equation (\ref{E:2d-coeffs}). This results in the slow decay of
$\widetilde{\frac{\partial}{\partial\tau}\mathcal{R}f}(\tau,\omega)$ for large
values of $\tau.$ Although we are not interested in those values, the slow
decay affects the accuracy of computation if the FFT is used to compute the
Fourier transform. We chose a brute force approach of over-sampling in the
frequency domain, which is equivalent to zero-padding $\widetilde{\frac
{\partial}{\partial\tau}\mathcal{R}f}(\tau,\omega)$ in $\tau.$ Since the
computation is done using the FFT, and only a small fixed number of terms (9
in our examples) is over-sampled, this does not compromise the speed of the algorithm.

\noindent{\textbf{Implementation and simulations.}} We illustrate performance
of our technique in a couple of numerical simulations. As a phantom
representing function $f(x)$ we used a linear combination of several slightly
smoothed characteristic functions of circles, as shown in
Figure~\ref{F:2D-phantom}. The data acquisition curve $S$ was the part of the
unit circle lying under the line $x_{2}=\sqrt{2}/2;$ the region $\Omega_{0}$
was the lower half of the unit disk, see Figure~\ref{F:2D-phantom}. Such an
acquisition geometry is a particular case of the open circular geometry
defined above (see Corollary \ref{T:geomcirc}), with $\mu=\pi/4.$
\begin{figure}[h]
\begin{center}
\includegraphics[scale = 0.5]{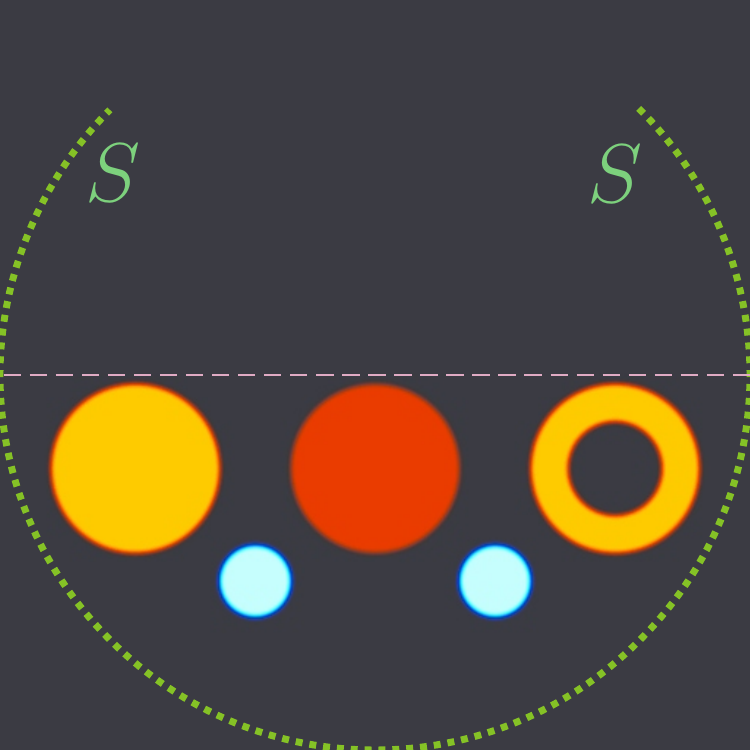} \phantom{a}
\includegraphics[scale = 0.45]{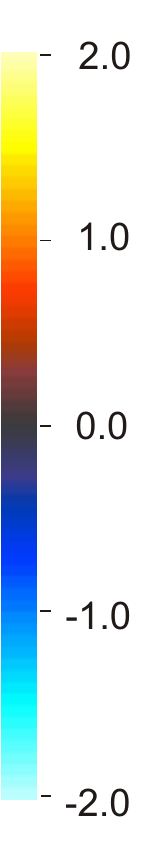}
\end{center}
\par
\vspace{-5mm} \caption{2D phantom; the dotted line shows the detector
locations $S,$ the thin dashed line is the upper boundary of the region
$\Omega_{0}.$ }%
\label{F:2D-phantom}%
\end{figure}The detector locations were modeled by sampling the unit circle
using 512 equispaced points. We evaluated the phantom on a fine 2048$\times
2048$ Cartesian grid, and computed the solution of the wave equation at the
grid's nodes, on the uniform grid in $t$ with 257 nodes covering the interval
$[0,2].$ Then this values were interpolated (spatially) to the detector
locations. The resulting solution of the wave equation on the 257$\times512$
grid on $[0,2]\times\lbrack0,2\pi]$ represents full data $g(t,\hat{y}(\psi)).$
It is shown in Figure~\ref{F:rec2D}(a). The accurate values of the Radon
projections were precomputed from $f(x)$ for future comparison on a uniform
257$\times512$ grid on $[-1,1]\times\lbrack0,2\pi]$. These Radon projections
are presented in Figure~\ref{F:rec2D}(b). In order to simulate the reduced
data $\widetilde{g}(t,\hat{y}(\psi))$ we replaced the values corresponding to
$\psi\in\lbrack\pi/4,3\pi/4]$ by zeros. In addition, we introduced a smooth
cut-off in $t,$ by multiplying $\widetilde{g}(t,\hat{y}(\psi))$ by a
$C^{6}[0,2]$ function $\chi(t),$ identically equal to 1 on the interval
$[0,1.3],$ and equal to 0 on the interval $[1.4,2].$ The resulting
$\widetilde{g}(t,\hat{y}(\psi+\pi))$ is shown in Figure~\ref{F:rec2D}(c); the
shift by $\pi$ helps to compare this image to that in part (a). The values of
$\widetilde{\mathcal{R}f}(\tau,\omega)$ obtained after the step 4 of our
algorithm are demonstrated in Figure~\ref{F:rec2D}(d), with the error (the
difference between $\widetilde{\mathcal{R}f}(\tau,\omega)$ and exact
$\mathcal{R}f(\tau,\omega))$ presented in the part (e). The wavy curve in the
latter image shows the boundary of the error-free part of
$\widetilde{\mathcal{R}f}(\tau,\omega)$ according to
Corollary~\ref{T:geomcirc}. The error under this curve is clearly small
compared to that above the curve. After step 5 we reconstructed projections
that closely coincide with the exact values, shown in part (b) of the figure.
In fact, the relative error in $L^{\infty}$ norm was about 5.0$\cdot$10$^{-4}$
in this experiment. The point-wise error is presented in Figure~\ref{F:rec2D}%
(f). \begin{figure}[h]
\begin{center}
\subfigure[Data $g(t,\hat{y}(\psi+\pi))$]{
\includegraphics[width=2.5in, height=1.47in]{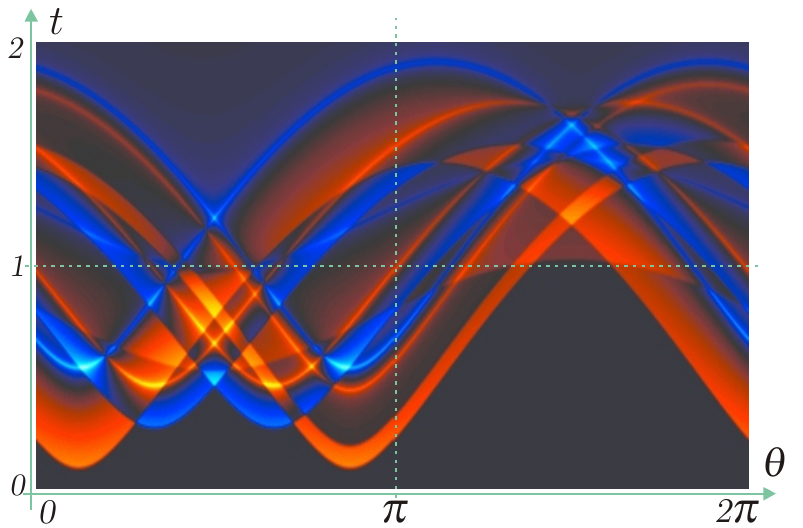}
\includegraphics[scale=0.42]{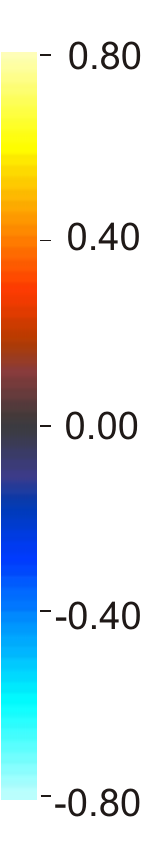} }
\subfigure[Exact Radon projections $Rf(\tau,\omega(\varpi))$]{
\includegraphics[width=2.5in, height=1.47in]{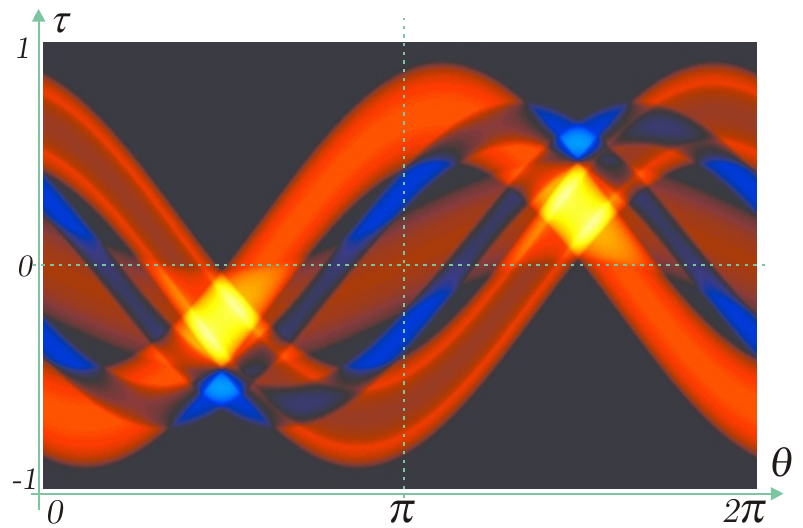}
\includegraphics[scale=0.42]{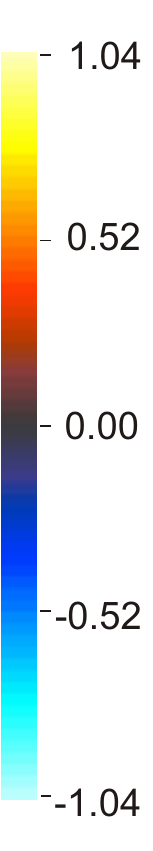} }
\subfigure[Reduced data $\widetilde{g}(t,\hat{y}(\psi+\pi))$]{
\includegraphics[width=2.5in, height=1.47in]{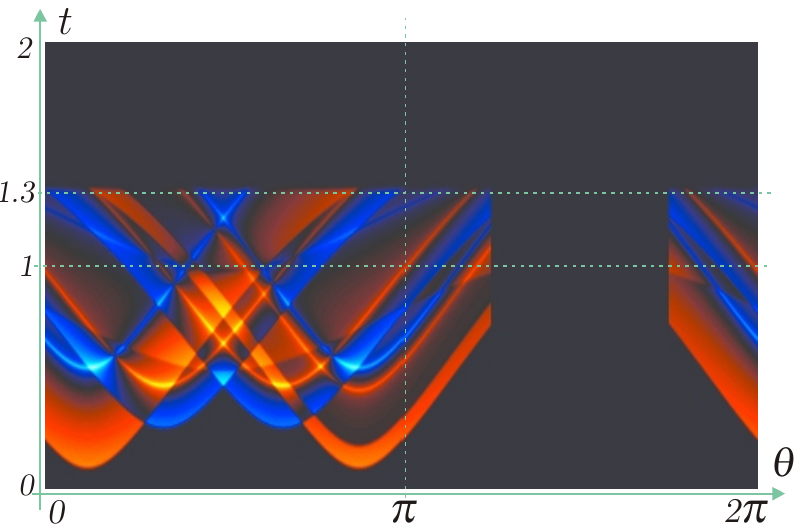}
\includegraphics[scale=0.42]{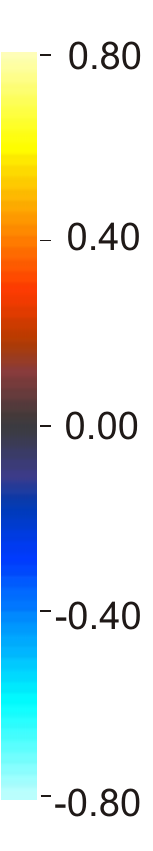} }
\subfigure[Projections $\widetilde{Rf}(\tau,\omega(\varpi))$ reconstructed on step 4]{
\includegraphics[width=2.5in, height=1.47in]{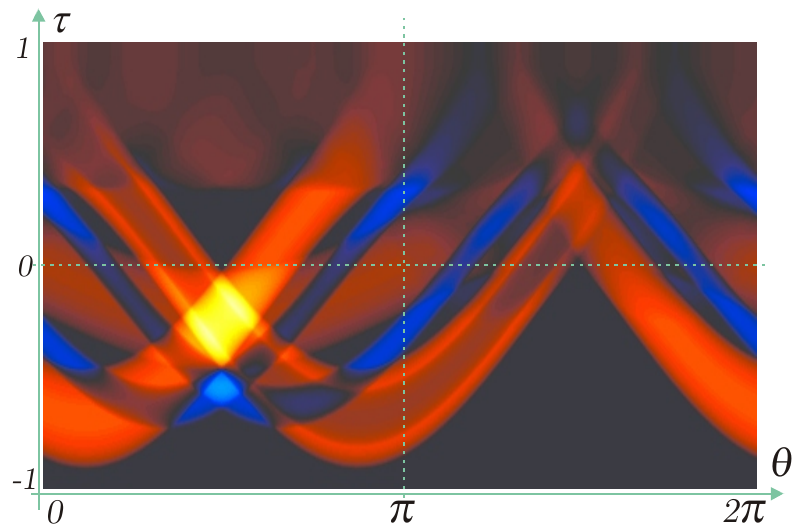}
\includegraphics[scale=0.42]{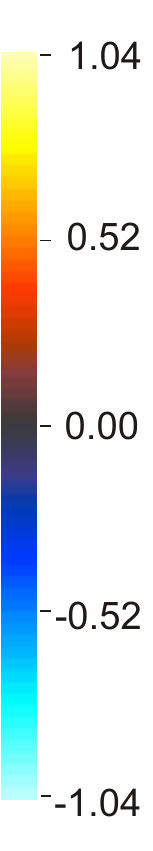} }
\subfigure[Reconstruction error after step 4]{
\includegraphics[width=2.5in, height=1.47in]{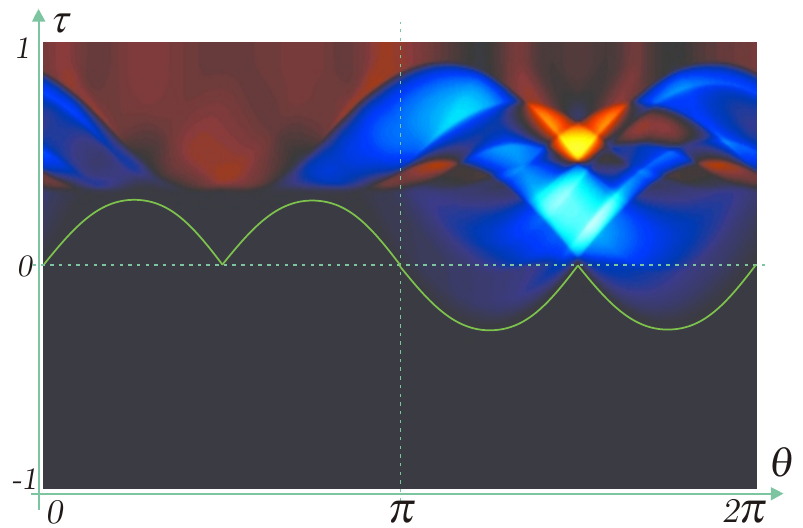}
\includegraphics[scale=0.42]{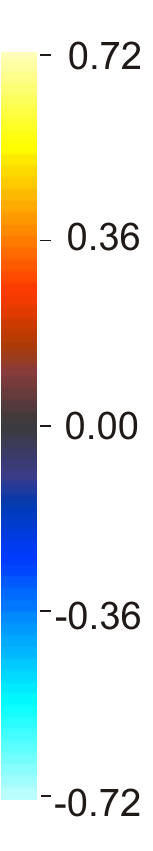} }
\subfigure[Reconstruction error after step 5]{
\includegraphics[width=2.5in, height=1.47in]{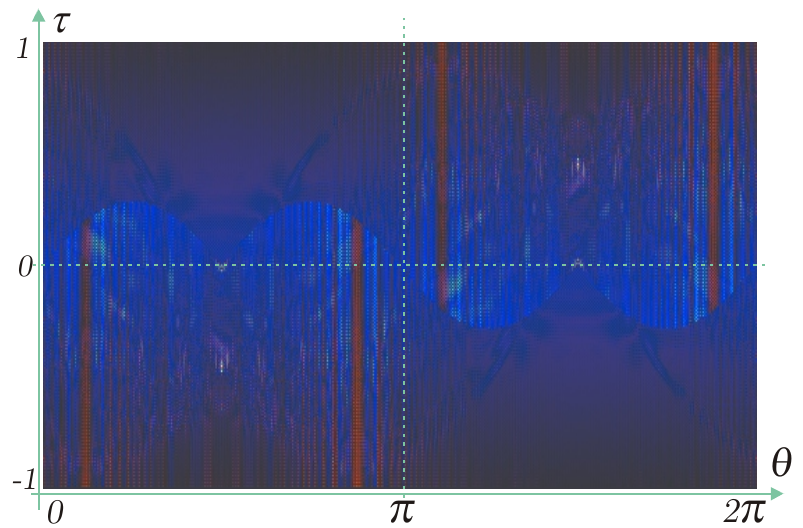}
\includegraphics[scale=0.42]{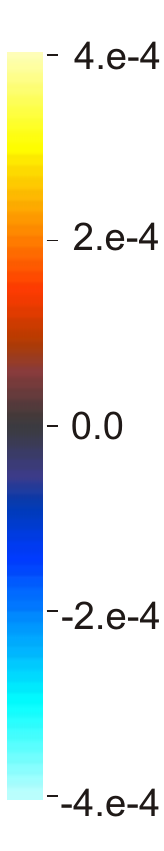} }
\end{center}
\par
\vspace{-5mm} \caption{Example of reconstruction from spatially and temporally
reduced data}%
\label{F:rec2D}%
\end{figure}In this paper we have investigated theoretically neither the
stability of the problem of finding the Radon projections, nor the stability
of our algorithms. We leave these topics for the future work. Instead, we will
demonstrate the stability of our approach numerically. We have modified the
previous simulation by contaminating the reduced data with normally
distributed spatially uncorrelated noise. The intensity of noise is $50\%$ in
$L^{2}$ norm. The noisy reduced data are shown in~Figure~\ref{F:recnoi2D}(a).
Perhaps, a better understanding of the level of noise can be gained from
Figure~\ref{F:recnoi2D}(c), where the graph of $\widetilde{g}(t,\hat{y}(0))$
is compared with the exact function. The result of the reconstruction is
demonstrated in~Figure~\ref{F:recnoi2D}(b); it should be compared against
Figure~\ref{F:rec2D}(b). One of the reconstructed projections, namely
$Rf(\tau,\omega(0))$, is plotted against the exact profile
in~Figure~\ref{F:rec2D}(d). The relative error in the reconstructed
projections (arising mostly due to the noisy data) is 7\% in $L^{2}$ norm.
This demonstrates a significant noise suppressing by the operator that
transforms the wave data into the projections.

\begin{figure}[h]
\begin{center}
\subfigure[Noisy reduced data $\widetilde{g}(t,\hat{y}(\psi+\pi))$]{
\includegraphics[width=2.5in, height=1.47in]{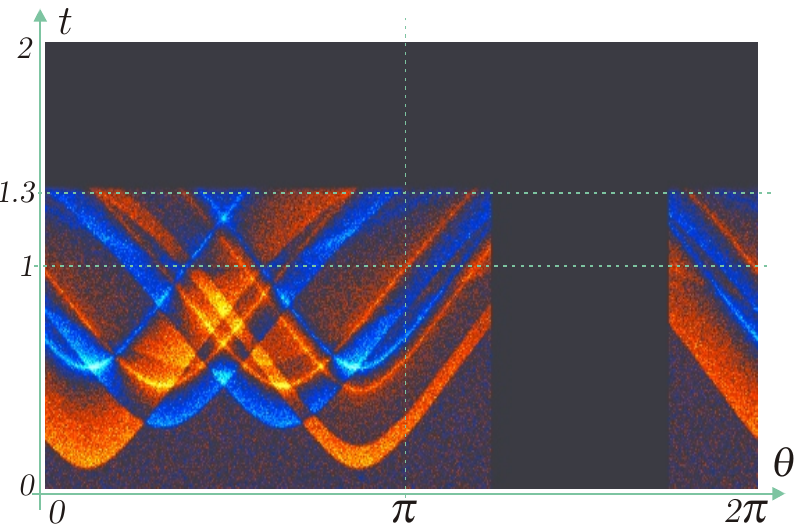}
\includegraphics[scale=0.42]{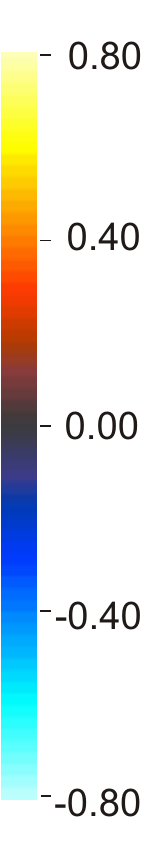} }
\subfigure[Reconstruction from noisy data]{
\includegraphics[width=2.5in, height=1.47in]{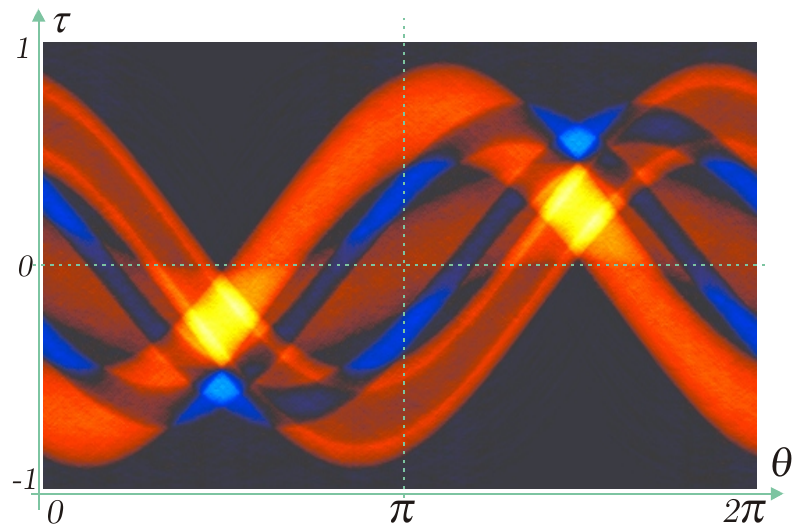}
\includegraphics[scale=0.42]{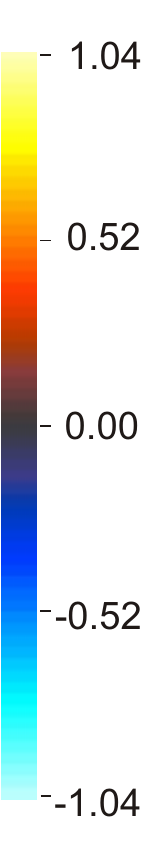} }
\subfigure[Noisy reduced data $\widetilde{g}(t,\hat{y}(0))$ \emph{vs} exact]{
\includegraphics[width=2.5in, height=1.0in]{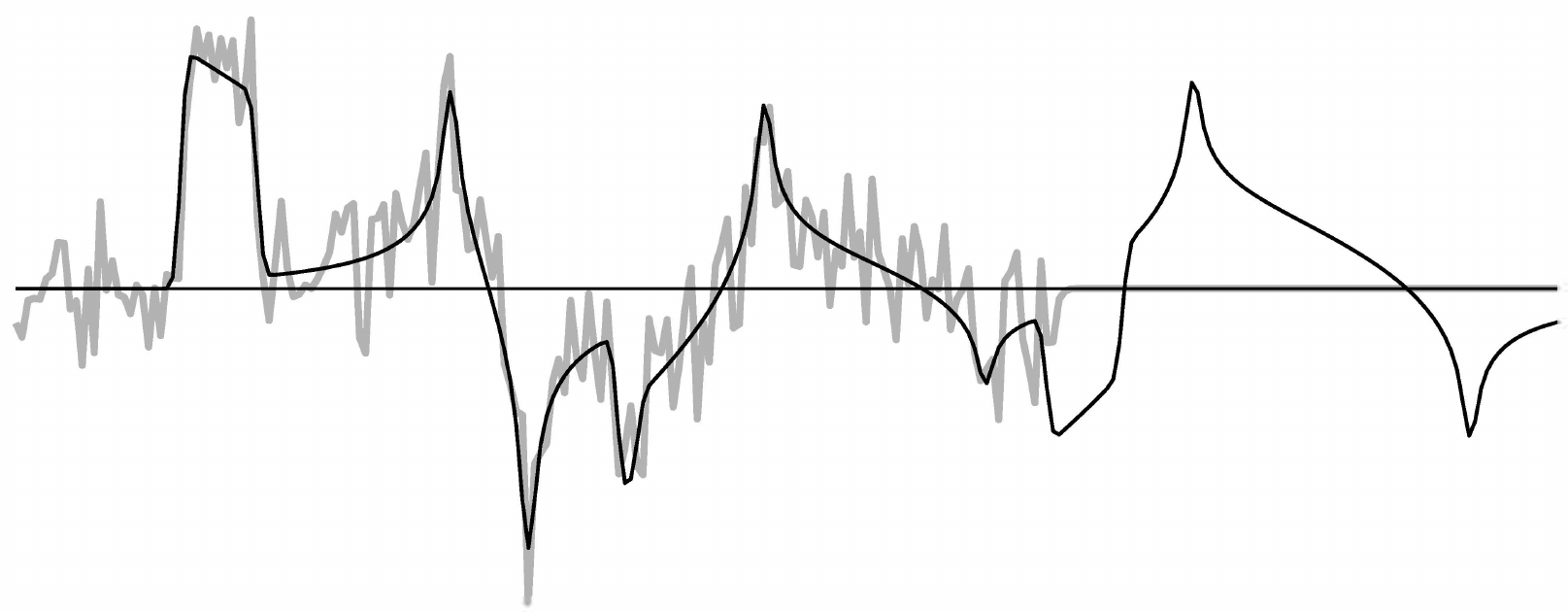} \phantom{aaa} }
\subfigure[Reconstructed $Rf(\tau,\omega(0))$ \emph{vs} exact]{
\includegraphics[width=2.5in, height=1.0in]{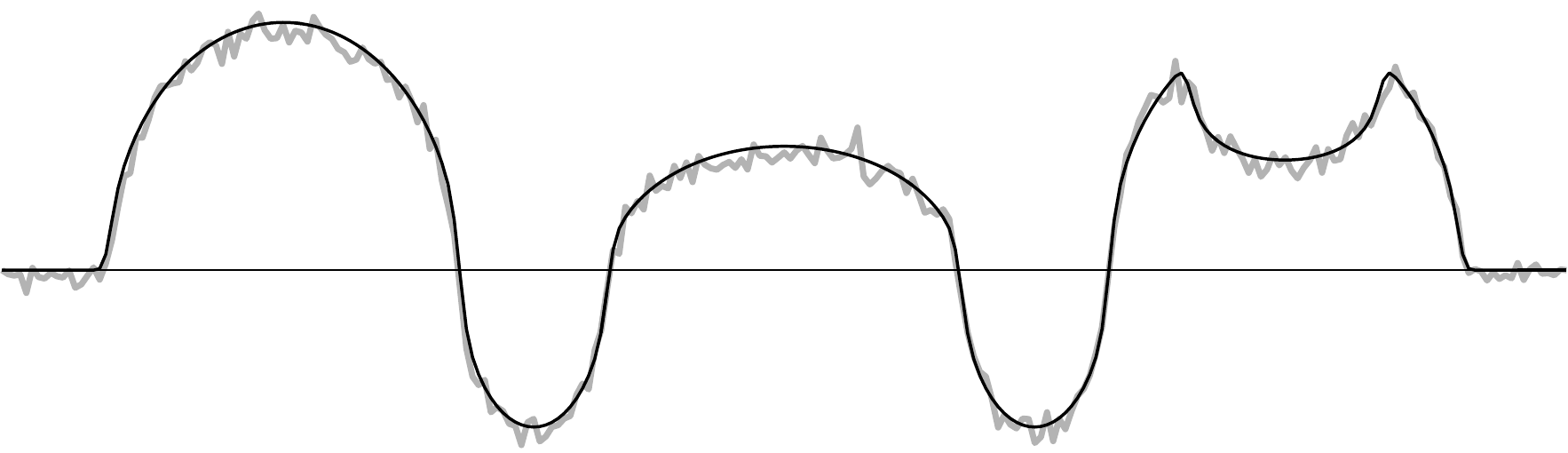}  \phantom{aaa} }
\end{center}
\par
\vspace{-5mm} \caption{Reconstruction from noisy reduced data}%
\label{F:recnoi2D}%
\end{figure}

\subsection{3D: spherical and open spherical geometries}

\subsubsection{Explicit expression for the density}

Let us find an explicit expression for the density $\varphi_{\omega}(\tau
,\hat{y})$ of the single layer potential representing a propagating delta wave
in 3D. The derivation is quite similar to that done above for the density in
the 2D case. The addition theorem for the Hankel function $h_{0}^{(1)}$ uses
spherical harmonics $Y_{k}^{m}(\hat{z}),$ $\hat{z}\in\mathbb{S}^{2},$
normalized so that%
\[
\int\limits_{\mathbb{S}^{2}}Y_{k}^{m}(\hat{z})\overline{Y_{k^{\prime}%
}^{m^{\prime}}(\hat{z})}d\hat{z}=\delta_{k,k^{\prime}}\delta_{m,m^{\prime}},
\]
where $\delta_{k,k^{\prime}}$ is the Kronecker symbol. Notice that%
\begin{equation}
Y_{k}^{m}(-\omega)=(-1)^{k}Y_{k}^{m}(\omega),\quad k=0,1,2,...\quad m=-k,..,k.
\label{E:Y-symm}%
\end{equation}
Then, the addition theorem for $h_{0}^{(1)}$ takes the following
form\cite{Colton}:%
\begin{equation}
h_{0}^{(1)}(\rho|y-x|)=4\pi\sum\limits_{k=0}^{\infty}\sum\limits_{m=-k}%
^{k}h_{k}^{(1)}(\rho|y|)j_{k}(\rho|x|)\overline{Y_{k}^{m}(\hat{y})}Y_{k}%
^{m}(\hat{x}),\quad y=\hat{y}|y|,\quad x=\hat{x}|x|, \label{E:addition-3d}%
\end{equation}
where $j_{k}$ and $h_{k}^{(1)}$ are, respectively, the spherical Bessel and
Hankel functions, $|y|>|x|,$ and $\rho\geq0.$

Consider now a plane wave $e^{i\rho\omega\cdot x},$ $x\in\mathbb{R}^{3}.$ It
can be expanded in spherical harmonics (using (\ref{E:Y-symm})\ and equation
(2.57) in \cite{Newton}) as follows%
\begin{equation}
e^{i\rho\omega\cdot x}=4\pi\sum\limits_{k=0}^{\infty}\sum\limits_{m=-k}%
^{k}i^{k}\overline{Y_{k}^{m}(\omega)}j_{k}(\rho|x|)Y_{k}^{m}(\hat{x}%
),\quad\rho\geq0. \label{E:time-harm3D}%
\end{equation}
For each fixed $\rho,$ spherical waves $j_{k}(\rho|x|)Y_{k}^{m}(\hat{x})$ in
the above equation can be represented in a form of a single layer potential of
the Helmholtz equation, supported on a unit sphere. The following formula is
easily verified using the addition theorem (equation (\ref{E:addition-3d})):%
\begin{equation}
\frac{1}{i\rho}\int\limits_{\mathbb{S}^{2}}\frac{1}{h_{k}^{(1)}(\rho)}%
Y_{k}^{m}(\hat{y})\hat{\Phi}_{3}(\rho,x-\hat{y})d\hat{y}=j_{k}(\rho
|x|)Y_{k}^{m}(\hat{x}),\quad|x|<1,\quad\rho\geq0. \label{E:single3D}%
\end{equation}
By combining (\ref{E:time-harm3D}) and (\ref{E:single3D}) one obtains:%
\begin{equation}
\int\limits_{\mathbb{S}^{2}}\left[  \frac{4\pi}{i\rho}\sum\limits_{k=0}%
^{\infty}\sum\limits_{m=-k}^{k}\frac{i^{k}\overline{Y_{k}^{m}(\omega)}%
Y_{k}^{m}(\hat{y})}{h_{k}^{(1)}(\rho)}\right]  \hat{\Phi}_{3}(\rho,x-\hat
{y})d\hat{y}=e^{i\rho\omega\cdot x},\quad|x|<1,\quad\rho\geq0.
\label{E:3D-ugly}%
\end{equation}
The interchange of order of integration and summation that took place in
deriving the above formula, is justified by the fast growth of $h_{k}%
^{(1)}(\rho)\equiv\sqrt{\pi/2\rho}H_{k}^{(1)}(\rho)$ with $k\rightarrow
\infty,$ due to $|H_{k}^{(1)}(\rho)|\thicksim\sqrt{2/\pi k}\left(  \frac
{2k}{e\rho}\right)  ^{k},$ (see formulas 9.3.1 in \cite{Abramowitz}). Let us
formally define function $\hat{\varphi}_{\omega}(\rho,\hat{y})$ as the
expression in brackets in (\ref{E:3D-ugly}) for $\rho\geq0,$ and as its
conjugate for $\rho<0$:%
\begin{equation}
\hat{\varphi}_{\omega}(\rho,\hat{y})\equiv\frac{4\pi}{i\rho}\sum
\limits_{k=0}^{\infty}\sum\limits_{m=-k}^{k}\frac{(-i)^{k}\overline{Y_{k}%
^{m}(\omega)}Y_{k}^{m}(\hat{y})}{h_{k}^{(1)}(\rho)},\quad\rho\geq0;\qquad
\hat{\varphi}_{\omega}(-\rho,\hat{y})\equiv\overline{\hat{\varphi}_{\omega
}(\rho,\hat{y})}. \label{E:def-dens-3D}%
\end{equation}
By computing a complex conjugate of (\ref{E:3D-ugly}), one obtains
representation for waves $e^{i\rho\omega\cdot x}$ with negative $\rho,$ so
that the following formula holds in 3D for all real $\rho$:
\begin{equation}
\int\limits_{\mathbb{S}^{2}}\hat{\varphi}_{\omega}(\rho,\hat{y})\hat{\Phi}%
_{3}(\rho,x-\hat{y})d\hat{y}=e^{i\rho\omega\cdot x},\quad|x|<1,\quad\rho
\in\mathbb{R}. \label{E:gen-plane-3D}%
\end{equation}

As in the 2D case, we first obtain single layer potentials representing smooth
approximations to the delta wave $\delta(\tau-\omega\cdot x),$ and then pass
to the limit. Consider a delta approximating family $\eta_{\varepsilon}(t).$
By multiplying (\ref{E:gen-plane-3D}) with $\frac{1}{2\pi}\hat{\eta
}_{\varepsilon}(\rho)e^{-i\rho\tau},$ integrating over $\mathbb{R}$, and using
(\ref{E:smooth-wave}), one obtains%
\[
\int\limits_{\mathbb{S}^{2}}\left(  \frac{1}{2\pi}\int\limits_{\mathbb{R}%
}\left[  \hat{\varphi}_{\omega}(\rho,\hat{y})\hat{\eta}_{\varepsilon}%
(\rho)\right]  \hat{\Phi}_{3}(\rho,x-\hat{y})e^{-i\rho\tau}d\rho\right)
d\hat{y}=\eta_{\varepsilon}\left(  \tau-\omega\cdot x\right)  ,\quad|x|<1.
\]
Since $\hat{\eta}_{\varepsilon}(\rho)$ decays at infinity faster than any
rational function of $\rho,$ the product $\hat{\varphi}_{\omega}(\rho,\hat
{y})\hat{\eta}_{\varepsilon}(\rho)$ can be Fourier transformed. Define
$\varphi_{\omega,\varepsilon}(\tau,\hat{y})$ on $\mathbb{R}\times
\mathbb{S}^{2}$ as follows%
\[
\varphi_{\omega,\varepsilon}(\tau,\hat{y})=\mathcal{F}^{-1}(\hat{\eta
}_{\varepsilon}(\rho)\hat{\varphi}_{\omega}(\rho,\hat{y}))(\tau,\hat{y}),
\]
Then, plane wave $\eta_{\varepsilon}\left(  \tau-\omega\cdot x\right)  $ is
represented by a single layer potential with density $\varphi_{\omega
,\varepsilon}(\tau,\hat{y})$
\[
\eta_{\varepsilon}\left(  \tau-\omega\cdot x\right)  =\int\limits_{\mathbb{S}%
^{2}}\left(  \int\limits_{\mathbb{R}}\varphi_{\omega,\varepsilon}(t,\hat
{y})\Phi_{3}(\tau-t,x-\hat{y})dt\right)  d\hat{y}.
\]
For each fixed $\varepsilon,$ wave $\eta_{\varepsilon}\left(  \tau-\omega\cdot
x\right)  $ is in the form (\ref{E:planewave}) and satisfies (\ref{E:finitesp}%
). Therefore, due to Proposition \ref{T:exist}, density $\varphi
_{\omega,\varepsilon}(t,\hat{y})$ vanishes on $W\cap(\mathbb{R\times S}^{2}),$
where $W$ is defined by (\ref{E:def-W}), with $\mathcal{T}(\omega)=(-1,1).$

As in the 2D case,\ in the limit $\varepsilon\rightarrow0$, plane waves
$\eta_{\varepsilon}\left(  \tau-\omega\cdot x\right)  $ converge to
$\delta\left(  \tau-\omega\cdot x\right)  .$ Repeating the argument presented
at the end of Section 5.1.1, one finds that the limit distribution
$\varphi_{\omega}(t,\hat{y})\equiv\lim_{\varepsilon\rightarrow0}%
\varphi_{\omega,\varepsilon}(t,\hat{y})$ vanishes on $W\cap(\mathbb{R\times
S}^{2}).$ This permits one to reconstruct Radon projections from the data
reduced both spatially and temporally, using formula (\ref{E:reduced}) or
(\ref{E:reduced-g}).

\subsubsection{Efficient 3D algorithm and simulations}

In this section we develop an efficient algorithm for the spatially reduced
spherical geometry. It is structurally similar to the 2D algorithm developed
in section \ref{S:fast2D}, and it reconstructs Radon projections
$\mathcal{R}f(\tau,\omega)$ from data $g(t,\hat{y})$ given on $[0,2-1/\sqrt
{2}]\times S.$ The computation is done in three stages: (I) reconstruction of
$\frac{\partial}{\partial\tau}\mathcal{R}f(\tau,\omega)$ using formula
(\ref{E:reduced-g}) for values of $\tau$ given by inequalities
(\ref{E:int-sphere}) (II) anti-differentiation in $\tau$ (III)\ reconstruction
the remaining projections using~(\ref{E:Radon-symm}).

As before, the first stage of the algorithm is a convolution, but performed
over $\mathbb{R\times S}^{2}$ instead of $\mathbb{R\times S}$. As in section
\ref{S:fast2D}, we extend $g(t,\hat{y})$ by zero to all $\hat{y}\in
\Gamma\backslash S,$ and to all $t\notin\lbrack0,2-1/\sqrt{2}],$ and denote
the extended data defined on $\mathbb{R\times S}^{2}$ by $\widetilde{g}%
(t,\hat{y}).$ Then (\ref{E:reduced-g})\ can be re-written in the form:
\begin{equation}
\widetilde{\frac{\partial}{\partial\tau}\mathcal{R}f}(\tau,\omega
)=\int\limits_{\mathbb{R}}\int\limits_{\mathbb{S}^{2}}\widetilde{g}(t,\hat
{y})\varphi_{\omega}(\tau-t,\hat{y})d\hat{y}dt, \label{E:sph-conv}%
\end{equation}
where function $\widetilde{\frac{\partial}{\partial\tau}\mathcal{R}f}%
(\tau,\omega)$ is equal to 0 for all $\tau\leq-1;$ it coincides with
$\frac{\partial}{\partial\tau}\mathcal{R}f(\tau,\omega)$ on intervals
(\ref{E:int-sphere}). Formula (\ref{E:sph-conv}) represents a convolution over
$\mathbb{R\times S}^{2}$ that can be computed using the Fourier-transform in
time and expansion in spherical harmonics in the spatial variables.

Equation (\ref{E:def-dens-3D}) yields the expansion of the kernel
$\varphi_{\omega}(t,\hat{y})$ in spherical harmonics/Fourier components. The
remaining details are as follows. We compute the Fourier transform of
(\ref{E:sph-conv}):%

\[
\widehat{\widetilde{\frac{\partial}{\partial\tau}\mathcal{R}f}}(\rho
,\omega)=\int\limits_{\mathbb{R}}\left[  \int\limits_{\mathbb{S}^{2}}%
\int\limits_{\mathbb{R}}\widetilde{g}(t,\hat{y})\varphi_{\omega}(\tau
-t,\hat{y})dtd\hat{y}\right]  e^{i\rho\tau}d\rho=\int\limits_{\mathbb{S}^{2}%
}\hat{\varphi}_{\omega}(\rho,\hat{y})\widehat{\widetilde{g}}(\rho,\hat
{y})d\hat{y},
\]
For $\rho\geq0,$ using the definition of $\hat{\varphi}_{\omega}$ (see
(\ref{E:def-dens-3D})) one obtains:%
\[
\widehat{\widetilde{\frac{\partial}{\partial\tau}\mathcal{R}f}}(\rho
,\omega)=\int\limits_{\mathbb{S}^{2}}\frac{4\pi}{i\rho}\sum\limits_{k=0}%
^{\infty}\sum\limits_{m=-k}^{k}\frac{i^{k}\overline{Y_{k}^{m}(\omega)}%
Y_{k}^{m}(\hat{y})}{h_{k}^{(1)}(\rho)}\widehat{\widetilde{g}}(\rho,\hat
{y})d\hat{y},\quad\rho\geq0.
\]
Further, expanding $\widehat{\widetilde{g}}(\rho,\hat{y})$ in spherical
harmonics yields coefficients $\widehat{\widetilde{g}}_{m,k}(\rho)$:%
\[
\widehat{\widetilde{g}}_{m,k}(\rho)=\int\limits_{\mathbb{S}^{2}}Y_{k}^{m}%
(\hat{y})\widehat{\widetilde{g}}(\rho,\hat{y})d\hat{y},\quad k=0,1,2,...,\quad
m=-k,..,k,
\]
so that%
\begin{equation}
\widehat{\widetilde{\frac{\partial}{\partial\tau}\mathcal{R}f}}(\rho
,\omega)=\sum\limits_{k=0}^{\infty}\sum\limits_{m=-k}^{k}\frac{4\pi}{i\rho
}\frac{i^{k}\widehat{\widetilde{g}}_{m,k}(\rho)}{h_{k}^{(1)}(\rho)}%
\overline{Y_{k}^{m}(\omega)},\quad\rho\geq0. \label{E:3D-coeffs-new}%
\end{equation}

In order to discretize the above formulas, we define computational grids in
$t,$ $\tau,$ $\rho,$ $\hat{y},$ and $\omega,$ and assume that extended data
$\widetilde{g}(t,\hat{y})$ is sampled on the product grid in $t$ and $\hat
{y}.$ We utilize uniform grids in $t,$ $\tau,$ and $\rho,$ allowing us to
compute the Fourier transforms using the FFT. Vectors $\hat{y}$ and $\omega$
on $\mathbb{S}^{2}$ are parametrized using azimuthal angle $\theta$ and the
polar angle $\varphi.$ We use uniform grid in $\theta$ and Gaussian
discretization points in $\cos\varphi.$ The steps of the algorithm are:

\noindent\textbf{1.} Expand $\widetilde{g}(t,\hat{y})$ in spherical harmonics
in $\hat{y}$ and compute the Fourier transform in $t,$ thus, obtaining
$\widehat{\widetilde{g}}_{m,k}(\rho),$ $k=0,1,2,...,\quad m=-k,..,k;$

\noindent\textbf{2.} For each grid value of $\rho\geq0,$ compute coefficients
$b_{m,k}(\rho)\equiv\frac{4\pi}{i\rho}\frac{i^{k}}{h_{k}^{(1)}(\rho
)}\widehat{\widetilde{g}}_{m,k}(\rho)$, and extend to negative $\rho$ by
complex conjugation;

\noindent\textbf{3.} Sum up series $\sum\limits_{k=0}^{\infty}\sum
\limits_{m=-k}^{k}b_{m,k}(\rho)\overline{Y_{k}^{m}(\omega)}$ for each grid
value of $\omega$ and $\rho$ and compute the inverse Fourier transform
in\ $\rho$ for each $\omega.$ This corresponds to formula
(\ref{E:3D-coeffs-new}), and yields values of $\widetilde{\frac{\partial
}{\partial\tau}\mathcal{R}f}(\tau,\omega)$ on a product grid in $\tau$ and
$\omega$.

\noindent\textbf{4.} Anti-differentiate $\widetilde{\frac{\partial}%
{\partial\tau}\mathcal{R}f}(\tau,\omega)$ on intervals (\ref{E:int-sphere})
finding $\widetilde{\mathcal{R}f}(\tau,\omega);$

\noindent\textbf{5.} Compute $\mathcal{R}f(\tau,\omega)$ by extracting the
correct values of $\widetilde{\mathcal{R}f(\tau,\omega)}$ within the intervals
(\ref{E:int-sphere}), and by finding the remaining values using
(\ref{E:Radon-symm}).

\begin{figure}[h]
\begin{center}
\subfigure[Data $g(t,\hat{y}(\theta_0,\varphi))$, $\theta_0 \approx 69^\circ$]
{\includegraphics[scale=0.67]{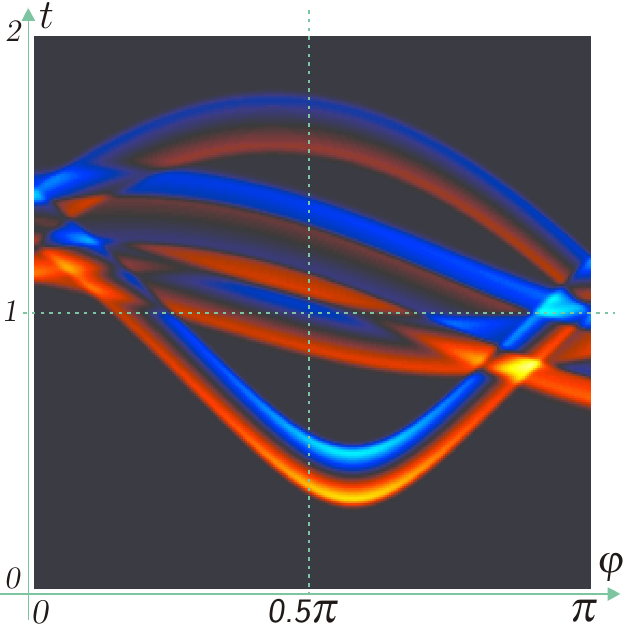}
\includegraphics[scale=0.47]{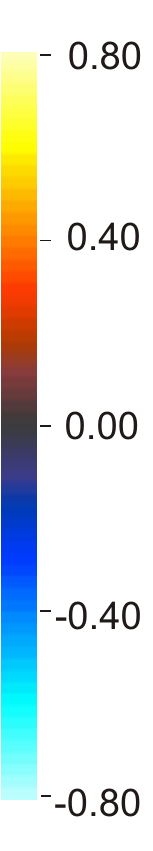} }
\subfigure[Reduced data $\widetilde{g}(t,\hat{y}(\theta_0,\varphi))$]{
\includegraphics[scale=0.67]{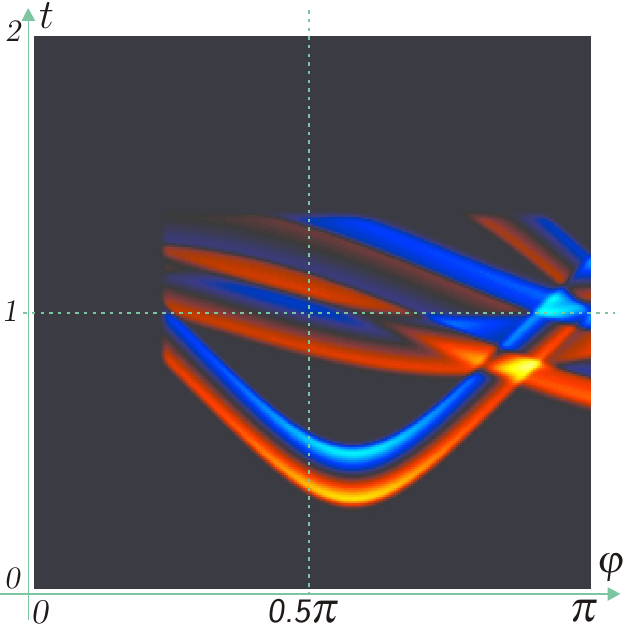}
\includegraphics[scale=0.47]{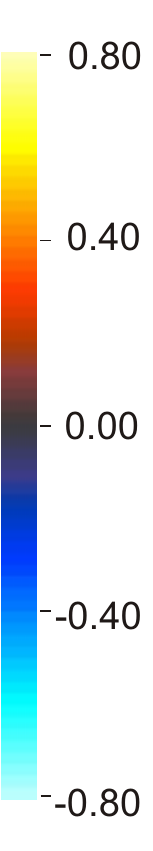} }
\subfigure[Reduced noisy data $\widetilde{g}(...)$]{
\includegraphics[scale=0.67]{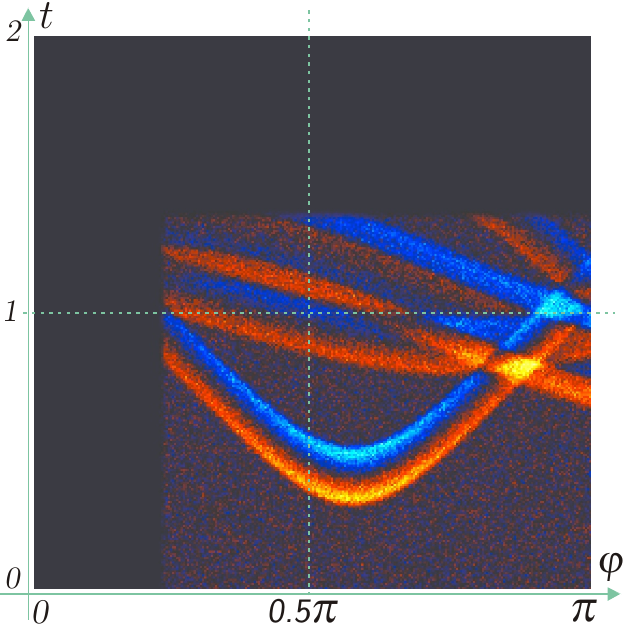}
\includegraphics[scale=0.47]{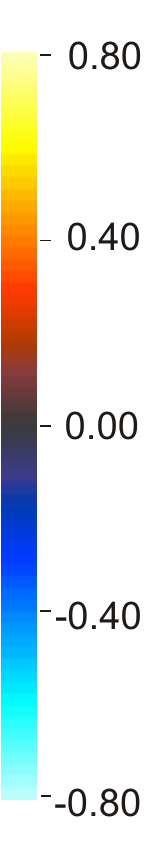} }
\par
\subfigure[Exact $Rf(\tau,\omega(\theta_0,\varphi))$, $\theta_0 \approx 69^\circ$]
{\includegraphics[scale=0.67]{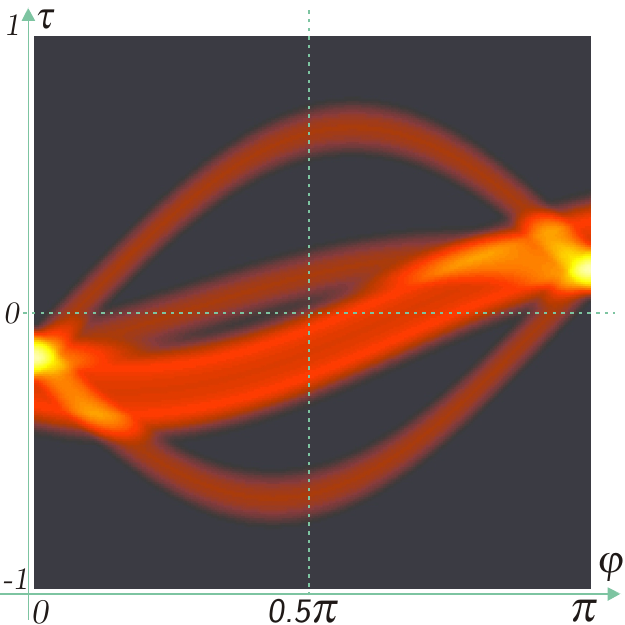}
\includegraphics[scale=0.47]{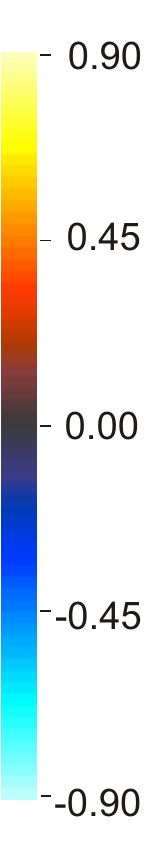} }
\subfigure[Step 4: $\widetilde{Rf}(\tau,\omega(\theta_0,\varphi))$]{
\includegraphics[scale=0.67]{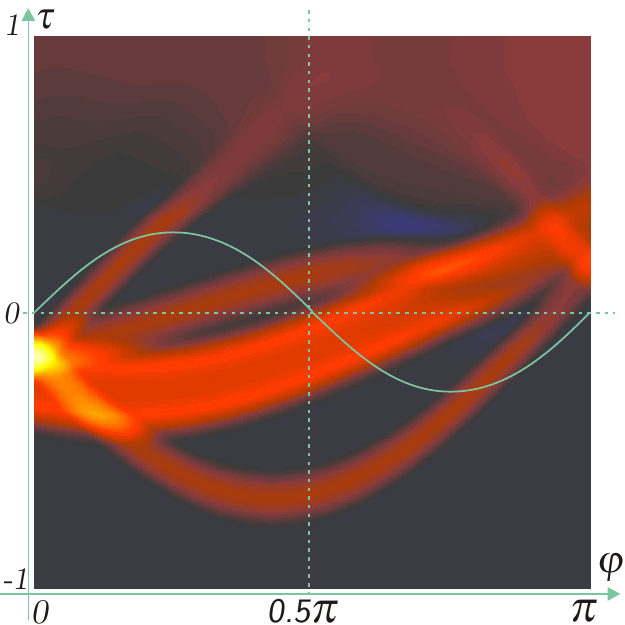}
\includegraphics[scale=0.47]{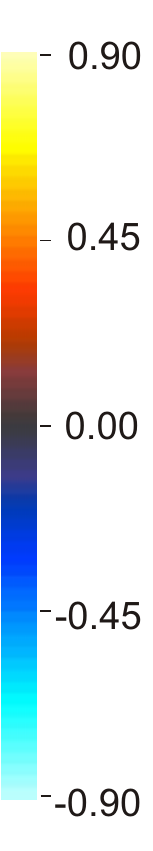} }
\subfigure[Step 5: error in $Rf(\tau,\omega(\theta_0,\varphi))$]{
\includegraphics[scale=0.67]{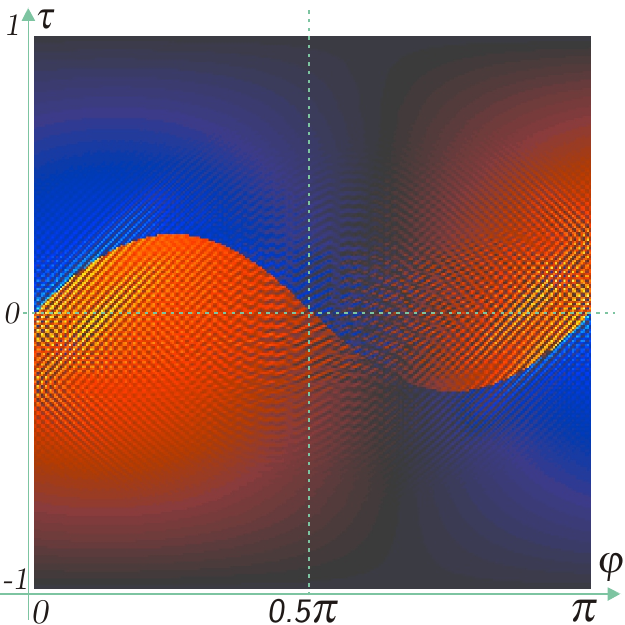}
\includegraphics[scale=0.47]{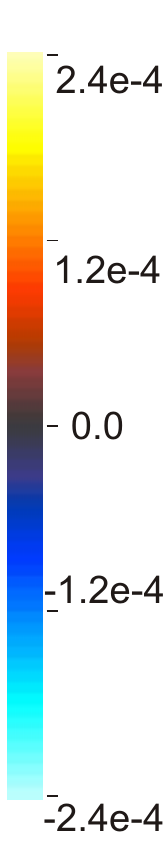} }
\par
\vspace{-5mm}
\end{center}
\caption{Example of a reconstruction from 3D spatially and temporally reduced
data}%
\label{F:rec3D}%
\end{figure}
The first three steps of the above algorithm coincide with the \emph{modified
Norton algorithm} for the spherical acquisition geometry \cite{Kun-cyl}. The
latter method is an efficient technique requiring full data $G(t,\hat{y}),$
known on $[0,2]\mathbb{\times S}^{2}$; it represents a development of
unaccelerated reconstruction algorithms first presented in
\cite{Norton1,Norton2}. Steps 4 and 5 of the present method are new; they
allow for the treatment of temporally and spatially reduced data that cannot
be handled by previous techniques. We refer the reader to \cite{Kun-cyl} for
the operation count and implementation details.

The above algorithm also works for the full spherical acquisition
($S=\Gamma\equiv\mathbb{S}^{2}$) with the temporally reduced data known on the
time interval $[0,1].$ The only modification is in the step 5, which we
replace with the step

\noindent\textbf{5}$^{\ast}$\textbf{.} Compute $\mathcal{R}f(\tau,\omega)$ by
extracting the correct values of $\widetilde{\mathcal{R}f}(\tau,\omega)$
within the interval $\tau\in\lbrack-1,0]$ for all $\omega\in\mathbb{S}^{2}$,
and finding remaining values using (\ref{E:Radon-symm}).

Interestingly, in contrast to the 2D case, no over-sampling in $\rho$ was
needed in the 3D\ case. While we do not have a rigorous explanation of this
phenomenon, the suspect is the different decay properties of solutions of the
wave equation in spaces of odd and even dimensions.

\noindent{\textbf{Implementation and simulations.}} We demonstrate performance
of our technique in a couple of numerical simulations, mirroring those
presented in Section \ref{S:fast2D}. The data acquisition curve $S$ was the
part of the unit sphere lying under the $x_{3}=\sqrt{2}/2;$ the region
$\Omega_{0}$ was the lower half of the concentric unit ball. As a phantom
representing function $f(x)$ we used a linear combination of several slightly
smoothed characteristic functions of balls, supported within $\Omega_{0}.$
Such an acquisition geometry is a particular case of the open spherical
geometry (see Corollary \ref{T:geomsphere}), with $\mu=\pi/4.$ The detector
locations $\hat{y}(\theta,\varphi)$ were modeled by sampling $(\theta
,\varphi)$\ on a product grid on the unit sphere, with uniformly distributed
512 points in $\theta$ and 401 Gaussian discretization points in $\cos
\varphi.$ Since in 3D a solution of the wave equation with full rotational
symmetry is given by an explicit formula, we were able to compute explicitly
$g(t,\hat{y}(\theta,\varphi))$ at each detector location on a uniform grid in
$t$ with 257 nodes covering the interval $[0,2].$ In Figure~\ref{F:rec3D}(a)
we exhibit a 2D slice of the data $g(t,\hat{y}(\theta,\varphi))$ corresponding
to the 100-th grid node in $\theta$, or $\theta_{0}\thickapprox69^{\circ}.$ In
order to simulate the reduced data $\widetilde{g}(t,\hat{y}(\theta,\varphi))$
we replaced the values corresponding to $\varphi\in\lbrack0,\pi/4]$ by zeros.
In addition, we introduced a smooth cut-off in $t,$ by multiplying
$\widetilde{g}(t,\hat{y})$ with a smooth function $\chi(t)$ defined in Section
\ref{S:fast2D}. The resulting $\widetilde{g}(t,\hat{y})$ is shown in
Figure~\ref{F:rec3D}(b); again, only a 2D slice corresponding to $\theta
_{0}\thickapprox69^{\circ}$ is presented. The values of
$\widetilde{\mathcal{R}f}(\tau,\omega)$ reconstructed on the 4-th step of our
algorithm can be seen in Figure~\ref{F:rec3D}(e), for $\omega=\omega
(\theta_{0},\varphi)$ with the same $\theta_{0}.$ The wavy line in the latter
figure indicates the upper boundary of the error-free region. The image in
Figure~\ref{F:rec3D}(e) can be compared to Figure~\ref{F:rec3D}(d) that shows
the exact values of the Radon projections we seek to reconstruct. After the
final, fifth step of the algorithm we obtain an image visually
indistinguishable of that in the part~(d) of the figure. The relative
reconstruction error in $L^{\infty}$ norm was about 3$\cdot10^{-4}$ in this
simulation. A 2D slice of that error, corresponding to $\theta=\theta_{0}$ is
shown in Figure~\ref{F:rec3D}(f).

We also conducted a simulation involving the same test function $f(x)$ and the
same acquisition geometry but with added noise in the reduced data
$\widetilde{g}(t,\hat{y})$. The level of the noise was 50\% in the $L^{2}$
norm. A 2D slice of noisy data is presented in Figure~\ref{F:rec3D}(c). The
noise in the reconstructed Radon projections was suppressed significantly; the
relative error was just under 1\% in the $L^{\infty}$ norm and less than 0.8\%
in $L^{2}$ norm. This suggests that the operator mapping the wave data into
the Radon projections is even more smoothing in 3D than it is in 2D.

\noindent{\textbf{Acknowledgment.}} The authors would like to thank the
anonymous referees for the numerous suggestions that helped to improve the
manuscript. The second author is grateful for the partial support by the NSF
through the awards NSF/DMS-1211521 and NSF/DMS-1418772.

\end{document}